\documentclass[a4paper, 10pt, reqno]{amsart}

\usepackage[T1]{fontenc}
\usepackage[utf8]{inputenc}

\usepackage{amssymb}
\usepackage{amsmath}
\usepackage{amsfonts}
\usepackage{amsthm}
\usepackage{mathtools}

\usepackage{fixltx2e} 
\usepackage{isomath} 

\usepackage{graphicx, xcolor}

\usepackage{url}%
\usepackage{hyperref}%
\hypersetup{
  colorlinks=false,
  linkbordercolor=gray,
  citebordercolor=gray,
  urlbordercolor=gray}%

\usepackage{comment}
\usepackage{booktabs}
\usepackage{geometry}
\usepackage{pdflscape}
\usepackage{tikz,pgfplots}
\usepackage{color}
\usepackage{float}
\usepackage{enumitem}

\usepackage{caption}

\newcommand{\E}{\mathbb{E}}

\DeclareMathOperator{\Cov}{Cov}

\DeclareMathOperator{\Span}{Span}
\DeclareMathOperator{\Var}{Var}

\theoremstyle{plain}%
\newtheorem{theorem}{Theorem}[section]
\newtheorem{lemma}[theorem]{Lemma}
\newtheorem{proposition}[theorem]{Proposition}
\newtheorem{corollary}[theorem]{Corollary}

\theoremstyle{definition}
\newtheorem{definition}[theorem]{Definition}
\newtheorem{example}[theorem]{Example}

\theoremstyle{remark}
\newtheorem{remark}[theorem]{Remark}

\title{Monotonicity properties of exclusion sensitivity}
\author{Malin Palö Forsström}
\thanks{Chalmers University of Technology, Gothenburg, Sweden. E-mail: \texttt{palo@chalmers.se}}
\date{\today}

\usepackage{todonotes}

\definecolor{code_gray}{rgb}{0.5,0.5,0.5}
\definecolor{dark_gray}{rgb}{0.2,0.2,0.2}
\definecolor{gray}{rgb}{0.5,0.5,0.5}
\definecolor{light_gray}{rgb}{0.95,0.95,0.95}
\definecolor{dark_blue}{RGB}{29,29,102}
\definecolor{mid_blue}{RGB}{51,51,178}
\definecolor{light_blue}{RGB}{153,153,217}

\makeatletter



\usepackage{setspace}
\begin{document}

\maketitle

\begin{abstract}
In~\cite{bgs2013}, exclusion sensitivity and exclusion stability for symmetric exclusion processes on graphs were defined as natural analogues of noise sensitivity and noise stability in this setting. As these concepts were defined for any sequence of connected graphs, it is natural to study the monotonicity properties of these definitions, and in particular, of whether some graphs are in some sense more stable or sensitive than others. The main purpose of this paper is to answer one such question which was stated explicitly in~\cite{bgs2013}. In addition, we get results about the eigenvectors and eigenvalues of symmetric exclusion processes on complete graphs.
\end{abstract}

\vspace{4em}
\setcounter{secnumdepth}{1}
\setcounter{tocdepth}{1}
\tableofcontents

\section{Introduction}
The notions \textit{noise sensitivity} and \textit{noise stability} were first introduced in~\cite{schramm2000}, describing how sensitive a sequence of Boolean functions \( f_n \colon \{ 0,1\}^n \to \{ 0,1\} \) was to a particular kind of noise in the argument. The main type of noise considered was re-sampling each entry in the argument of \( f_n \) with a small probability, or equivalently, by letting each such entry run a continuous time Bernoulli process in a short time interval. Since this paper was published, similar definitions have been made in slightly different settings, by changing one or several of the elements in the setup, such as the domain of the functions \( f_n \), the range of the functions \( f_n \) or the process constituting the noise (\cite{schramm2000,st1999,md2005,abgm2014, kd2012}). In~\cite{bgs2013}, the range and domain of the functions \( (f_n)_{n \geq 1} \) was kept from the original setting, but the process was changed into a symmetric exclusion process with respect to some sequence of connected graphs, \( (G_n)_{n \geq 1} \). In this new setting, it is natural to ask to what extent the sensitivity of a sequence of functions depends on the sequence of graphs. This is the main subject of this paper.

We now define what we mean by a symmetric exclusion process. 
Let \( (G_n)_{n \geq 1} \) be a  sequence of finite connected graphs and let \( (\alpha_n)_{n \geq 1} \) be a sequence of strictly positive real numbers. We are interested in the sequence of Markov chains \( (X^{(n)} )_{n \geq 1} \), where \( X^{(n)} \) is a \emph{symmetric exclusion process} on \( \{ 0,1 \}^{V(G_n)} \) with rate \( \alpha_n \). This process can be defined as follows. At time zero, put a black or white marble at each vertex of the graph. Now for each edge \( e \in E(G_n) \), associate an independent  Poisson clock with rate \( \alpha_n \). When this clock rings, interchange the marbles at the endpoints of \( e \). Let \( X_t^{(n)} \) be the configuration of marbles at time~\( t \).

In general, for a graph \( G_n \) we will identify configurations of black and white marbles with elements in \( \{ 0,1 \}^{V(G_n)} \) by letting the numbers be indicators of black marbles. Similarly, for each \( \ell \in \{ 0,1, \ldots, |V(G_n)| \} \), we will identify elements in \( \binom{V(G_n)}{\ell} \) with configurations with exactly \( \ell \) black marbles by representing such a configuration by the set of vertices at which there are black marbles. 

It is easy to check that the uniform distribution \( \pi_n \) on \( \{ 0,1 \}^{V(G_n)} \) will be a stationary distribution for the symmetric exclusion process \( X^{(n)} \) on \( G_n \) with rate \( \alpha_n \). However, as the Markov process \( X^{(n)} \) is not irreducible, this is not the only stationary distribution, and in fact for any \( \ell \in \{ 0, \ldots, |V(G_n)| \}  \), the uniform distribution \( \pi_n^{(\ell)} \) on \( \binom{V(G_n)}{\ell} \) will be a stationary distribution for~\( X^{(n)} \). For \( x,y \in \binom{V(G)}{\ell} \), write \( x \sim y \) to denote that \( y \) can be obtained from \( x \) by interchanging the marbles at the endpoints of some edge in \( E(G) \). Whenever we pick \( X_0^{(n)} \) according to \( \pi_n \), we will say that \( X^{(n)} \) is a symmetric exclusion process with respect to \( (G_n, \alpha_n)_{n \geq 1} \), and whenever we pick \( X_0^{(n)} \) according to \( \pi_n^{(\ell_n)} \) we will say that \( X^{(n)} \) is a symmetric exclusion process with respect to \( (G_n, \alpha_n, \ell_n )_{n \geq 1} \).

We now give the two definitions from~\cite{bgs2013} with which we will be concerned.
\begin{definition}
Let \( (G_n)_{ n \geq 1 } \) be a sequence of finite connected graphs and let \( (\alpha_n)_{n\geq 1 } \) be a sequence of real numbers. For each \( n \geq 1 \), let \( X^{(n)} \) be the symmetric exclusion process on \( \{ 0,1 \}^{V(G_n)} \) with rate \( \alpha_n \) where \( \mathcal{L}(X_0^{(n)} ) =  \pi_n \).
The sequence of functions \( f_n \colon \{ 0,1 \}^{V(G_n)} \to \{0,1 \} \) is said to be \emph{exclusion sensitive} (XS) with respect to \( ( G_n, \alpha_n )_{n \geq 1} \) if 
\[
\lim_{n \to \infty}\Cov (f_n(X_0^{(n)}), f_n(X_1^{(n)})) = 0
\]
\end{definition}

The next definition captures an opposite behavior.
\begin{definition}
Let \( (G_n)_{ n \geq 1 } \) be a sequence of finite connected graphs and let \( (\alpha_n)_{n\geq 1 } \) be a sequence of real numbers. For each \( n \geq 1 \), let \( X^{(n)} \) be the symmetric exclusion process on \( \{ 0,1 \}^{V(G_n)} \) with rate \( \alpha_n \) where \( \mathcal{L}(X_0^{(n)} )= \pi_n \).
The sequence of functions \( f_n \colon \{ 0,1 \}^{V(G_n)} \to \{0,1 \}  \) is said to be \emph{exclusion stable} (XStable) with respect to \( ( G_n, \alpha_n )_{n \geq 1} \) if 
\[
\lim_{\varepsilon \to 0} \limsup_n P((f(X_0^{(n)}) \not =  f(X_\varepsilon^{(n)})) = 0,
\]
or equivalently, if
\[
\lim_{\varepsilon \to 0} \limsup_n \E[(f(X_0^{(n)}) - f(X_\varepsilon^{(n)}))^2] = 0.
\]
\end{definition}

Now let \( K_n \) be the complete graph on \( n \) vertices.  In addition to the two definitions above, a sequence of functions \( (f_n)_{n \geq 1} \), where \( f_n \colon \{ 0,1 \}^{V(G_n)} \to \{0,1 \}  \), is said to be \emph{complete graph exclusion sensitive} (CGXS) is it is exclusion sensitive with respect to \( ( K_n, 1/n)_{n \geq 1} \) and \emph{complete graph exclusion stable} (CGXStable) if it is exclusion stable with respect to \( ( K_n, 1/n)_{n \geq 1} \).

It is relatively easy to see that any sequence of functions \( (f_n)_{n \geq 1} \) for which \( {\lim_{n \to \infty} \Var (f_n(X_0^{(n)})) = 0} \) will be both exclusion stable and exclusion sensitive with respect to \( (G_n, \alpha_n)_{n \geq 1} \) for any sequence \( (\alpha_n)_{n \geq 1 } \) of positive numbers. For this reason, we will only be interested in so called \textit{nondegenerate} sequences of functions \( (f_n)_{n \geq 1} \), meaning that \( \Var(f_n(X_0^{(n)})) \) is uniformly bounded away from zero.

In~\cite{bgs2013}, the authors asked the following two questions. If \( (G_n)_{n \geq 1} \) is a sequence of graphs, \( \alpha_n \) satisfies \( \alpha_n \leq 1/\max_{v \in V(G_n)} \deg v \) and \( f_n \colon \{ 0,1 \}^{V(G_n)} \to \{ 0,1 \} \) is a sequence of functions, is it the case that
\begin{enumerate}
\item \( (f_n)_{n \geq 1} \) is XS with respect to \( (G_n, \alpha_n)_{n \geq 1} \) \( \Rightarrow\) \( (f_n)_{n \geq 1} \) is CGXS?
\item \( (f_n)_{n \geq 1} \) is CGXStable \(\Rightarrow\) \( (f_n)_{n \geq 1} \) is XStable with respect to \( (G_n, \alpha_n)_{ n\geq 1} \)?
\end{enumerate}
The main objective of this paper is to provide a proof of the following result, which provides a positive answer to both questions.
\begin{theorem}
Let \( (f_n)_{ n \geq 1 } \), \( f_n \colon \{ 0,1 \}^n \to \{0,1 \}  \), be a sequence of  functions and let \( (G_n)_{n \geq 1} \) be a sequence of connected graphs. Further let \( \alpha_n \leq  1/ \max_{v \in V(G_n)} \deg v \).  Then
\begin{enumerate}
\item[(i)] if \( (f_n)_{n \geq 1} \) is exclusion sensitive with respect to \( (G_n, \alpha_n )_{n \geq 1} \), then \( (f_n)_{n\geq 1 } \) is exclusion sensitive with respect to \( (K_{|V(G_n)|}, 1/|V(G_n)|)_{n \geq 1} \).
\item[(ii)] if \( (f_n)_{n \geq 1} \) is exclusion stable with respect to \( (K_{|V(G_n)|}, 1/|V(G_n)|)_{n \geq 1} \), then \( (f_n)_{n\geq 1 } \) is exclusion stable with respect to \( (G_n, \alpha_n)_{n \geq 1} \).
\end{enumerate}
\label{theorem:main result}
\end{theorem}

Given the positive answers to both  questions above, one might ask if being exclusion sensitive (or exclusion stable) is monotone with respect to adding edges to the graphs \( (G_n )_{n \geq 1} \). We will later see that this is true if we use the same rates \( (\alpha_n) \) for both graphs, but the following example shows that if we only use the  restriction on the rates from the previous theorem, that is if we only assume that for each sequence of graphs, \( \alpha_n \leq  1/ \max_{v \in V(G_n)} \deg v \), we will not always get monotonicity.
\begin{example}\label{example: first example}
 Let \( G_n \) be the graph with vertex set \( \{ 1,2, \ldots, 2n \} \) and an edge between two vertices \( i \) and \( j \) if and only if \( |i-j|=1 \mod 2n \).  Further, let \( G_n' \) be the graph obtained from \( G_n \) by adding an edge between each pair of vertices \( i ,j \in \{ n+1, n+2, \ldots, 2n\} \) that are not already connected by an edge. Finally, let \( G_n'' = K_{2n} \).
 
For \(x \in \{ 0,1\}^{V(G_n)}\), define \( f_n(x) = (-1)^{\left|x \cap \{ 1,3, \ldots n\}\right|} \). Then  \( (f_n)_{n \geq 1 } \)  is exclusion sensitive with respect to \( { (G_n, 1/2)_{n \geq 1}} \), exclusion stable with respect to \( { (G_n', 1/(n-1))_{n \geq 1} }\) and exclusion sensitive with respect to \(  {(G_n'', 1/(n-1))_{n \geq 1}} \). To see this, note that for an exclusion process on \( G_n \), at time \( \varepsilon \), about a proportion \( \varepsilon \) of the clocks on edges in the upper circle will have ticked. It follows that in the limit, almost surely the number of ones at vertices labeled with 1,3, \ldots, \( n \) will have changed arbitrarily many times, rendering the sequence of functions exclusion sensitive with respect to \( { (G_n, 1/2)_{n \geq 1}} \). By contrast, for an exclusion process on \( G_n' \) with rate \( 1/(n-1) \), there is a positive probability that none of the marbles on the upper part of the graph will have moved, why in this setting the sequence \( (f_n) \) is exclusion stable. On the last graph, the rate for each edge is the same up to a constant as for \( G_n' \), but now there are enough edges connected to the odd labeled vertices in the upper part of the graph for the sequence of functions to be exclusion sensitive.

\begin{figure}[H]
\centering
\begin{minipage}[b]{0.3\linewidth}
\centering 

\begin{tikzpicture}[scale = 1]
  \def\N{13}
  \def\M{7}
  \def\L{7}

  \foreach \i in {0,...,\N}{
    \draw ({cos((\i+0.5)*360/(\N+1))},{sin((\i+0.5)*360/(\N+1))}) -- ({cos((\i+1.5)*360/(\N+1))},{sin((\i+1.5)*360/(\N+1))});
  };

  \foreach \i in {0,...,\N}{
    \draw[fill = white] ({cos((\i+0.5)*360/(\N+1))},{sin((\i+0.5)*360/(\N+1))})  circle (2pt);
  };

  \foreach \i in {1,3,...,\M}{
    \draw ({1.3*cos((\M-\i+0.5)*360/(\N+1))},{1.3*sin((\M-\i+0.5)*360/(\N+1))})  node {\small \i};
 };
 \end{tikzpicture}

\caption*{(a) The graph \( G_7\).}
\end{minipage}
\hspace{-0em}
\begin{minipage}[b]{0.3\linewidth}
\centering
\begin{tikzpicture}[scale = 1]
  \def\N{13}
  \def\M{7}
  \def\L{7}
  
  \foreach \i in {0,...,\N}{
    \draw ({cos((\i+0.5)*360/(\N+1))},{sin((\i+0.5)*360/(\N+1))}) -- ({cos((\i+1.5)*360/(\N+1))},{sin((\i+1.5)*360/(\N+1))});
  };

  \foreach \i in {\L,...,\N}{
    \foreach \j in {\i,...,\N}{
    \draw ({cos((\i+0.5)*360/(\N+1))},{sin((\i+0.5)*360/(\N+1))}) -- ({cos((\j+0.5)*360/(\N+1))},{sin((\j+0.5)*360/(\N+1))});
    };
  };

  \foreach \i in {0,...,\N}{
    \draw[fill = white] ({cos((\i+0.5)*360/(\N+1))},{sin((\i+0.5)*360/(\N+1))})  circle (2pt);
  };

  \foreach \i in {1,3,...,\M}{
    \draw ({1.3*cos((\M-\i+0.5)*360/(\N+1))},{1.3*sin((\M-\i+0.5)*360/(\N+1))})  node {\small \i};
 };

 \end{tikzpicture}
\caption*{(b) The graph \( G_{7}'\).}
\end{minipage}
\hspace{-0em}
\begin{minipage}[b]{0.3\linewidth}
\centering
\begin{tikzpicture}[scale = 1]
  \def\N{13}
  \def\M{7}
  \def\L{7}

  \foreach \i in {0,...,\N}{
    \foreach \j in {\i,...,\N}{
    \draw ({cos((\i+0.5)*360/(\N+1))},{sin((\i+0.5)*360/(\N+1))}) -- ({cos((\j+0.5)*360/(\N+1))},{sin((\j+0.5)*360/(\N+1))});
    };
  };

  \foreach \i in {0,...,\N}{
    \draw[fill = white] ({cos((\i+0.5)*360/(\N+1))},{sin((\i+0.5)*360/(\N+1))})  circle (2pt);
  };

  \foreach \i in {1,3,...,\M}{
    \draw ({1.3*cos((\M-\i+0.5)*360/(\N+1))},{1.3*sin((\M-\i+0.5)*360/(\N+1))})  node {\small \i};
 };

 \end{tikzpicture}
\caption*{(c) The graph \( G_{7}''\).}
\end{minipage}
\caption{The three graphs described in Example~\ref{example: first example} when \( n = 7 \).}
\end{figure}

\end{example}

In the previous example, one thing that made monotonicity fail was that the structure of the graphs were different enough to suggest that very different rates should be used for the corresponding exclusion processes. It is therefore natural to ask whether using the same rates before and after adding edges would be a strong enough assumption to get monotonicity.

We now state our second main result, which shows that for any fixed sequence of rates \( (\alpha_n)_{n\geq 1} \), any sequence of graphs \( (G_n)_{n\geq 1} \) and any sequence of functions \( f_n \colon V(G_n) \to \{0,1 \}   \), the properties of being exclusion sensitive and exclusions stable with respect to \( (G_n, \alpha_n )_{n\geq 1} \) are monotone with respect to adding edges to the graphs in \( (G_n)_{n\geq 1 } \).

\begin{theorem}
Let \( (G_n)_{ n \geq 1} \) and \( (G_n')_{n\geq 1} \) be two sequences of finite connected graphs with \( V(G_n) = V(G_n') \) and \( E(G_n') \subseteq E(G_n) \), and let \( (\alpha_n) \) be a sequence of strictly positive real numbers. Let \( f_n \colon \{ 0,1 \}^{V(G_n)} \to \{0,1 \}  \) be a nondegenerate sequence of functions. Then
\begin{enumerate}
\item[(i)] if \( (f_n)_{n \geq 1} \) is XS with respect to \( (G_n', \alpha_n )_{n \geq 1} \), then \( (f_n)_{n \geq 1} \) is XS with respect to \( (G_n, \alpha_n)_{n \geq 1} \). 
\item[(ii)] if \( (f_n)_{n \geq 1} \) is XStable with respect to \( (G_n, \alpha_n )_{n \geq 1} \), then \( (f_n)_{n \geq 1} \) is XStable with respect to \( (G_n', \alpha_n)_{n \geq 1} \). 
\end{enumerate}
\label{proposition: monotonicity}
\end{theorem}

For more general versions of this theorem, see Remarks~\ref{remark: disconnected}~and~\ref{remark: different rates}.


The rest of this paper will be structured as follows. In the next section, we give notation for the eigenvectors and eigenvalues for the different processes with which we will be concerned in the rest of this paper. We also give spectral equivalences of exclusion stability and exclusion sensitivity. In the third section we study the structure of these eigenvectors and eigenvalues a bit more closely. In particular, we prove some results concerning these when we have a symmetric exclusion process with respect to \( (K_n, \alpha_n)_{n\geq 1} \) for some sequence \( (\alpha_n)_{n\geq 1} \) of real numbers. In the fourth section, we give a proof of our first main result, Theorem~\ref{theorem:main result}. Finally, in the last section, we give a proof Theorem~\ref{proposition: monotonicity}.

\section{Spectral equivalences of exclusion sensitivity and stability}

All of the results presented in this paper will use methods from  Fourier analysis. In this section we will define the functions which we will use as a basis, and derive some simple results.

Let \( X^{(n, \ell)} \) be a symmetric exclusion process with respect to \( (G_n, \alpha, \ell) \). Then \( X^{(n, \ell)} \) is a Markov process and has a generator \( Q_n^{(\ell)} \). For functions \( f,g \colon \binom{V(G_n)}{\ell} \to \mathbb{R}\), \( \langle f, g \rangle = \langle f, g\rangle_{\pi_n^{(\ell)}} \coloneqq \E[f(X^{(n, \ell)}_0)g(X^{(n, \ell)}_0)] \) is an inner product. As \( X^{(n,\ell)} \) is  reversible and irreducible, we can find a set \( \{ \psi_ i^{(n,\ell)}\}_i \) of eigenvectors of \( { -Q_n^{(\ell)}} \), with corresponding eigenvalues 
\begin{equation}
0 = \lambda_1^{(n,\ell)} < \lambda_2^{(n, \ell)} \leq \lambda_3^{(n, \ell)} \leq \ldots \leq \lambda_{\binom{|V(G_n)|}{\ell}}^{(n, \ell)}
\end{equation}
such that  \( \{ \psi_ i^{(n,\ell)}\}_i \) is an orthonormal basis with respect to \( \langle \cdot, \cdot \rangle \) for the space of real valued functions on \( \binom{V(G_n)}{\ell} \). Note  that we can assume that \( \psi_1^{(n , \ell) } \equiv 1 \) for all \( \ell \) and \( n \).

Next, for all \( t \geq 0 \), let \( \smash{H_t^{(n,\ell)}}  \)  denote the continuous time Markov semigroup given by
\[
H_t^{(n, \ell)} = \exp(tQ_n^{(\ell)}) .
\]
In other words, \( H_t^{(n,\ell)} \) operates on a function \( f \) with domain \( \binom{V(G_n)}{\ell}  \) by
\[
H_t^{(n, \ell)} f(x) = \E[f(X_t^{(n, \ell)}) \mid X_0^{(n, \ell)} = x].
\]

The eigenvectors \( \{ \psi_ i^{(n, \ell)}\}_i \) will be eigenvectors of \( H_t^{(n, \ell)} \) as well, with corresponding eigenvalues \( \{ e^{-\lambda_i^{(n, \ell)} t} \}_i \).
Since the set \( \{ \psi_ i^{(n, \ell )}\}_i \) is an orthonormal basis, for any \( f \colon \binom{V(G_n)}{\ell} \to \mathbb{R} \) we can write
\[
f(x) = \sum_{i=1}^{\binom{|V(G_n)|}{\ell}} \langle f, \psi_i^{(n, \ell)} \rangle \, \psi_i^{(n, \ell)} (x).
\]

To simplify notations, we will write  \( \hat f^{(\ell)}(i) \) instead of \( \langle f, \psi_i^{(n, \ell)} \rangle \). Using these Fourier coefficients, for any function \( f \colon \binom{V(G_n)}{\ell} \to \mathbb{R}\) we have that
\[
\E[f(X_0^{(n, \ell)})] = \langle f, 1 \rangle = \langle f, \psi_1^{(n, \ell)} \rangle = \hat f^{(\ell)} (1)
\]
and
\[
\Var(f(X_0^{(n, \ell)})) = \langle f, f \rangle - \hat f^{(\ell)}(1)^2 = \sum_{i \geq 2} \hat f^{(\ell)}(i)^2.
\]

Another well known (see e.g. (1.8) on page 5 in~\cite{chung1996}) characterization of the eigenvalues \(\{\lambda_i^{(n,\ell)}\}_i\) which will be useful for us later is
\begin{equation}
\lambda_i^{(n, \ell)}
= 
\min_{f \colon \langle f, \psi_{i'}^{(n, \ell)} \rangle =0\textnormal{ for all } i' < i } \frac{\langle -Q_n^{(\ell)}f, f \rangle}{\langle f, f \rangle},
\label{equationIV: rayleigh quotient}
\end{equation}
where  the minimum is attained by the corresponding eigenvector \( \psi_i^{(n, \ell)} \). The ratio on the right  hand side of~\eqref{equationIV: rayleigh quotient} is called the Rayleigh quotient of \( -Q_n^{(\ell)} \). It is easy to see that if \( \psi \) is an eigenvector of \( -Q_n^{(\ell)} \), then the Rayleigh quotient is the corresponding eigenvalue.


Using the definition of the generator \( Q_n^{(\ell)} \), we can write
\begin{align}
{2\langle -Q_n^{(\ell)}f, f \rangle} &= {2\alpha \!\!\!\!\!\!\!\! \sum_{x,y \in \binom{V(G_n)}{\ell}\colon x \sim y} \!\!\!\!   \pi_n^{(\ell)}(x)  f(x)(f(x)-f(y))} \nonumber
\\&= {\alpha \!\!\!\!\!\!\!\!  \sum_{x,y \in \binom{V(G_n)}{\ell} \colon x \sim y}\!\!\!\!   \pi_n^{(\ell)}(x)  (f(x)-f(y))^2} \nonumber
\\&= { \alpha \, {\textstyle\binom{|V(G_n)|}{\ell}^{-1}} \!\!\!\!\!\!\!\!  \sum_{x,y \in \binom{V(G_n)}{\ell} \colon x \sim y}\!\!\!\!     (f(x)-f(y))^2}
. \label{equation: Rayleigh quotient as sum}
\end{align}
It follows that if \( Q_n \) is the generator of the exclusion process with respect to \( (G_n, \alpha, \ell) \) and \( Q_n' \) the  generator of the exclusion process with respect to \( (G_n', \alpha, \ell) \) for some graph \( G_n' \) satisfying \( V(G_n)=V(G_n') \) and \( E(G_n') \subseteq E(G_n) \), then for any function \( f \colon \binom{V(G_n)}{\ell} \to \mathbb{R} \),
\begin{equation}
\frac{\langle -Q_n' f, f \rangle}{\langle f, f \rangle}
\leq
\frac{\langle -Q_n f, f \rangle}{\langle f, f \rangle},
\label{equation: Q inequality}
\end{equation}
as for the right hand side of this inequality, the sum in~\eqref{equation: Rayleigh quotient as sum} simply contains more terms.

Above, we listed some simple properties of the eigenvectors of the generator \( Q_n^{(\ell)} \) of an exclusion process with a fixed number of black marbles. The next lemma relates these eigenvectors to  the  eigenvectors of the generator \( Q_n \).

\begin{lemma}
Let \( G_n\) be a  finite connected graph and let \( \alpha_n\) be a strictly positive real number. For each \({ \ell \in \{ 0,1,\ldots, |V(G_n)| \} }\), let \( Q_n^{(\ell)} \) be the generator of the exclusion process with respect to \( (G_n, \alpha_n, \ell) \), and let  \( \{ \psi_i^{(n,\ell)}\}_{i} \) be an orthonormal basis of eigenvectors of \( -Q_n^{(\ell)} \) with corresponding eigenvalues \( \{ \lambda_i^{(n,\ell)}\}_{i} \). Define \( \psi_{i,\ell}^{(n)} \colon \{ 0,1 \}^n \to \mathbb{R} \) by
\begin{equation}
 \psi_{i,\ell}^{(n)}(x) \coloneqq \begin{cases} \sqrt{\frac{2^n}{\binom{n}{\ell}}} \cdot  \psi_i^{(n,\ell)}(x)  &\textnormal{if } x \in \binom{V(G_n)}{\ell}  \cr 0 &\textnormal{otherwise.} \end{cases}
\label{equation: full basis}
\end{equation}
Then \( \{   \psi_{i,\ell}^{(n)}\}_{i, \ell} \) is an orthonormal basis for \( -Q_n \), where \( Q_n \) is the generator of the symmetric exclusion process with respect to \( (G_n, \alpha_n)\), with corresponding eigenvalues \( \{   \lambda_{i}^{(n,\ell)}  \}_{i, \ell}\).
\label{lemma: general eigenvectors}
\end{lemma}

\begin{proof}
Let \( x \in \{ 0,1 \}^n \) and suppose first that \( x \in \binom{V(G_n)}{\ell} \). Then
\begin{equation*}
\begin{split}
-Q_n \psi_{i,\ell}^{(n)}(x) 
&= 
\sum_{y \sim x} \alpha (\psi_{i,\ell}^{(n)}(x) - \psi_{i,\ell}^{(n)}(y))
\\&=
\sqrt{\frac{2^n}{\binom{n}{\ell}}} \cdot  \sum_{y \sim x} \alpha (\psi_{i}^{(n,\ell)}(x) - \psi_{i}^{(n,\ell)}(y)) 
\\&=
\sqrt{\frac{2^n}{\binom{n}{\ell}}} \cdot  \left( -Q_n^{(\ell)} \psi_i^{(n,\ell)}(x)  \right)
\\&=
\sqrt{\frac{2^n}{\binom{n}{\ell}}} \cdot  \lambda_i^{(n,\ell)}\psi_i^{(n,\ell)}(x) 
\\&=
  \lambda_i^{(n,\ell)}\psi_{i,\ell}^{(n)}(x).
\end{split}
\end{equation*}
On the other hand, if \( x \in \binom{V(G_n)}{\ell} \), then by definition, \(  \psi_{i,\ell}^{(n)}(x)  = 0 \), why clearly,
\begin{equation*}
-Q_n \psi_{i,\ell}^{(n)}(x) 
= 
0
=
  \lambda_i^{(n,\ell)} \cdot 0
=
  \lambda_i^{(n,\ell)}\psi_{i,\ell}^{(n)}(x).
\end{equation*}
Consequently, the equation
\begin{equation*}
-Q_n \psi_{i,\ell}^{(n)}(x) 
=
  \lambda_i^{(n,\ell)}\psi_{i,\ell}^{(n)}(x)
\end{equation*}
is valid for all \( x \in \{ 0,1 \}^n \). This shows that \( \psi_{i,\ell}^{(n)} \) is an eigenvector of \( -Q_n \) with corresponding eigenvalue \( \lambda_{i,\ell}^{(n)} = \lambda_i^{(n,\ell)} \). 

The claim of orthonormality follows similarly, and is therefore omitted here.
\end{proof}

\begin{remark}
Note that the eigenvectors given by~\eqref{equation: full basis} with eigenvalue equal to zero is independent of the chosen graph \( G_n  \) as long as \( G_n \) is connected.
\label{remark: equal eigenvectors}
\end{remark}

Below and in the rest of this paper, whenever \( f,g \colon \{ 0,1\}^{V(G_n)} \to \mathbb{R} \), we will write \( \langle f, g \rangle = \langle f, g \rangle_{\pi_n}  \coloneqq \E[f(X^{(n)}_0) g(X^{(n)}_0)]\) and whenever we calculate the Fourier coefficients of some real valued  function \( {f \colon \{ 0,1 \}^{V(G_n)} \to \mathbb{R}} \) with respect to the basis \( \{ \psi_{i, \ell} \}_{i, \ell} \) given by~\eqref{equation: full basis},  we write \( \hat f(i,\ell) \coloneqq \langle f, \psi_{i,\ell}^{(n)} \rangle \).
Also, we will for \( x \in \{ 0,1 \}^{V(G_n)} \) let \( \| x \| \coloneqq \sum_{v \in V(G_n)} x(v) \).

The next result provides a spectral characterization of what it means to be exclusion sensitive, and it is the equivalent definition it provides that we will use in all subsequent results. Together with Proposition~\ref{proposition: spectral XStability} this result  
is a complete analogue of Theorem 1.9 in~\cite{schramm2000}, and similar analogues for exclusion process, although for a different set of eigenvectors, can be found in~\cite{bgs2013} (Proposition~3.1 and Proposition~3.2 respectively). The proofs of both these results very similar to the proofs in the original setting (see e.g.~\cite{gs2014}) after conditioning on \( \| X_0^{(n)} \| \).

\begin{proposition}
Let \( (G_n)_{n \geq 1} \) be a sequence of finite connected graphs, let \( (\alpha_n)_{n \geq 1} \) be a sequence of positive real numbers and for each \( {n \geq 1} \), let \( X^{(n)} \) be the exclusion process with respect to \( (G_n, \alpha_n)_{n \geq 1} \).  Further let \( \{ \psi_{i,\ell}^{(n)} \}_{i,\ell} \) be the orthonormal basis of eigenvectors defined in Lemma~\ref{lemma: general eigenvectors}, and let \( \{ \lambda_{i, \ell}^{(n)}\}_{i,\ell} \)  be the corresponding eigenvalues. Then a sequence  \( f_n \colon \{ 0,1 \}^{V(G_n)} \to \{0,1 \}  \) is exclusion sensitive with respect to \( (G_n, \alpha_n )_{n \geq 1} \) if and only if
\begin{enumerate}
\item[(i)] \( \lim_{n \to \infty}\Var (\E[f(X_0^{(n)}) \mid  \| X_0^{(n)} \| ]) = 0 \) and
\item[(ii)] for all \( k > 0 \), 
\[
\lim_{n \to \infty} \sum_{i, \ell \colon 0 < \lambda_{i, \ell}^{(n)} \leq k } \hat f_n(i, \ell)^2 = 0.
\]
\end{enumerate}
\label{proposition: spectral XS} 
\end{proposition}

\begin{proof}
First recall the well known result stating that for three random variables \( X \), \( Y \) and \( Z \),
\[
\Cov (X,Y) = \E \left[\Cov(X,Y \mid Z) \right]+ \Cov \left( \E [X \mid Z], \E [Y \mid Z]\right).
\]
Fix \( \varepsilon > 0 \) and set \( X = f(X_0^{(n)}) \), \( Y = f(X_\varepsilon^{(n)})  \) and \( Z = \| X_0^{(n)} \|  \) to obtain
\begin{equation}
\begin{split}
&\Cov(f(X_0^{(n)} ),f(X_{\varepsilon}^{(n)})) 
\\&\hspace{3em}=
\sum_{\ell=0}^n P(\| X_0^{(n)} \| = \ell) \cdot \Cov \left( f(X_0^{(n)}), f(X_\varepsilon^{(n)}) \mid \| X_0^{(n)}\| = \ell \right)
\\&\hspace{16em}+
\Var \left(  \E \left[ f(X_0^{(n)}) \mid  \| X_0^{(n)} \|  \right] \right).
\end{split}
\label{equation: removing the variance}
\end{equation}
We will now rewrite the term  \( \Cov \left( f(X_0^{(n)}), f(X_\varepsilon^{(n)}) \mid \| X_0^{(n)}\| = \ell \right) \) in the expression above. To this end, note first that for any \( \ell \in \{ 0,1,\ldots, n \} \) and any \( t > 0 \),
\begin{equation*}
\begin{split}
\E \left[ f_n(X_0^{(n)})f_n(X_{t}^{(n)}) \right]
&=
\E \left[ f_n(X_0^{(n)}) \E \left[ f_n(X_{t}^{(n)})  \mid X_0^{(n)} \right]\right]
\\&=
\E \left[f_n(X_0^{(n)})H^{(n)}_{t} f_n(X_0^{(n)}) \right]
\\&=
 \left\langle f_n, H^{(n)}_{t} f_n\right\rangle.
\end{split}
\end{equation*}
Writing \( f_n \) as \( f_n = \sum_{i,\ell} \hat f_n(i,\ell) \psi_{i, \ell}^{(n)} \) it follows that
\begin{equation*}
\begin{split}
 \left\langle f_n, H^{(n)}_{t} f_n\right\rangle
&=
\left\langle \sum_{i,\ell}\hat f_n(i,\ell) \psi_{i,\ell}^{(n)}, \sum_{j,\ell'} \hat f_n(j,\ell') H^{(n)}_{t} \psi^{(n)}_{j,\ell'} \right\rangle
\\&=
 \left\langle
\sum_{i,\ell}\hat f_n(i,\ell) \psi_{i,\ell}^{(n)} ,\sum_{j,\ell'} \hat f_n(j,\ell') e^{ - t  \lambda_{j,\ell'}^{(n)} }\psi^{(n)}_{j,k'} 
\right\rangle
\\&=
\sum_{i,j,\ell,\ell'} e^{-t \lambda_{j,\ell'}^{(n)} } \hat f_n(i,\ell) \hat f_n(j,\ell')   \left\langle \psi_{i,\ell}^{(n)},\psi^{(n)}_{j,\ell'} \right\rangle.
\end{split}
\end{equation*}
As  \( \{ \psi_{i,\ell}^{(n)} \}_{i, \ell} \) is an orthonormal set, summing up, we obtain
\begin{equation}
\begin{split}
\E \left[ f_n(X_0^{(n)})f_n(X_{t}^{(n)}) \right]
=
\sum_{i,\ell} e^{-t \lambda_{i,\ell}^{(n)} } \hat f_n(i,\ell)^2 .
\end{split}
\label{equation: expected value of product}
\end{equation}

Using that 
\begin{align*}
\E \left[ f_n (X_0^{(n)}) \mid \|X_0^{(n)} \| = \ell \right] 
& = \sum_{x \in \binom{V(G_n)}{\ell}} \pi_n^{(\ell)}(x) f_n(x)
 = \sum_{x \in \binom{V(G_n)}{\ell}} \pi_n^{(\ell)}(x) f_n(x) \cdot 1
\\& = \sum_{x \in \binom{V(G_n)}{\ell}} \pi_n^{(\ell)}(x) f_n(x) \cdot \psi_1^{(n,\ell)}
 = \langle f_n, \psi_1^{(n,\ell)} \rangle
\\&= \hat f_{n,\ell}(1) 
= \left( P(\| X_0^{(n)} \| = \ell ) \right)^{-1/2}  \hat f_{n}(1,\ell)
\end{align*}
we now get
\begin{equation*}
\begin{split}
\sum_{\ell=0}^n P(\| X_0^{(n)} \| = \ell) \cdot \Cov \left( f(X_0^{(n)}), f(X_\varepsilon^{(n)} ) \mid \| X_0^{(n)}\| = \ell \right)\hspace{-22em}&
\\&=
\E \left[ f_n(X_0^{(n)})f_n(X_{\varepsilon}^{(n)}) \right] -\sum_{\ell = 0}^n P(\| X_0^{(n)} \| = \ell) \cdot \E \left[ f_n (X_0^{(n)}) \mid \| X_0^{(n)} \| = \ell\right]^2  
\\&= \sum_{\ell=0}^n \sum_{i = 2}^{\binom{|V(G_n)|}{\ell}} e^{-\varepsilon \lambda_{i,\ell}^{(n)} } \hat f_n(i,\ell)^2 .
\end{split}
\end{equation*}
Note in particular that the term \( e^{-\varepsilon \lambda_{i,\ell}^{(n)} } \hat f_n(i,\ell)^2  \) in the previous equation is positive.  From this fact and~\eqref{equation: removing the variance}, it follows that \( (f_n)_{n \geq 1 } \) can be XS with respect to \( (G_n, \alpha_n)_{n \geq 1} \) if and only if
\begin{enumerate}
\item[(i)] \( \lim_{n \to \infty} \Var \left(  \E \left[ f(X_0^{(n)}) \mid   \| X_0^{(n)} \|  \right] \right) = 0 \), and
\item[(ii')] \( \lim_{n \to \infty} \sum_{\ell=0}^n \sum_{i = 2}^{\binom{|V(G_n)|}{\ell}} e^{-\varepsilon \lambda_{i,\ell}^{(n)} } \hat f_n(i,\ell)^2 = 0 \)
\end{enumerate}

For any \( \varepsilon > 0 \), it is easy to see that (ii) is satisfied if and only if (ii') holds. From this the desired conclusion follows.
\end{proof}

The following result provides an analogue of Proposition~\ref{proposition: spectral XS} for exclusion stability.

\begin{proposition}
Let \( (G_n)_{n \geq 1} \) be a sequence of finite connected graphs, let \( (\alpha_n)_{n \geq 1} \) be a sequence of positive real numbers and for each \( n \geq 1 \), let \( X^{(n)} \) be the exclusion process with respect to \( (G_n, \alpha_n)_{n \geq 1} \).  Further let \( \{ \psi_{i,\ell}^{(n)} \}_{i,\ell} \) be the orthonormal basis of eigenvectors defined in Lemma~\ref{lemma: general eigenvectors}, and let \( \{ \lambda_i^{(n,\ell)}\}_{i,\ell} \)  be the corresponding eigenvalues.
Then a sequence \( f_n \colon \{0,1\}^{V(G_n)} \to \{0,1 \}  \) is exclusion stable with respect to \( (G_n, \alpha_n )_{n \geq 1} \) if and only if for all \( \delta > 0 \) there is \( k \in \mathbb{N} \) such that

\begin{equation}
\sup_{n} \sum_{i, \ell \colon \lambda_{i,\ell}^{(n)} \geq k} \hat f_{n}(i,\ell)^2 < \delta.
\end{equation}
\label{proposition: spectral XStability}
\end{proposition}

\begin{proof}
First note that since \( f_n \) is Boolean, we have that
\begin{equation*}
\begin{split}
\hspace{3em}&\hspace{-3em}P\left( f_n(X^{(n)}_{\varepsilon}) \not = f_n(X^{(n)}_0)\right)
\\&= 
\E \left[ f_n(X^{(n)}_{\varepsilon}) (1-f_n(X^{(n)}_0)) \right] 
\\&\hspace{2em}+ \E \left[ f_n(X^{(n)}_0) (1-f_n(X^{(n)}_{\varepsilon}))\right]
\\&= 2 \left( \E \left[ f_n(X^{(n)}_0) \right]  - \E \left[f_n(X^{(n)}_0)  f_n(X^{(n)}_{\varepsilon}) \right] \right) 
\\&= 2 \left( \E \left[ f_n(X^{(n)}_0) f_n(X^{(n)}_0)   \right]  - \E \left[f_n(X^{(n)}_0)  f_n(X^{(n)}_{\varepsilon}) \right] \right) .
\end{split}
\end{equation*}
Combining this with~\eqref{equation: expected value of product} for \( t = 0 \) and \( t = \varepsilon \), we obtain
\begin{equation*}
\begin{split}
P \left( f_n(X_0^{(n)}) \not = f_n(X_{\varepsilon}^{(n)}) \right)
&=
2 \left(\sum_{i,\ell} \hat f_n(i,\ell)^2  -  \sum_{i,\ell} e^{-\varepsilon \lambda_{i, \ell}^{(n)}} \hat f_n(i,\ell)^2\right)
\\&=
2 \sum_{i, \ell} \left(1-e^{-\varepsilon \lambda_{i,\ell}^{(n)}} \right) \hat f_n(i, \ell)^2.
\end{split}
\end{equation*}

For the \emph{if} direction of the proof, suppose that  for any \( \delta > 0 \) there  is \( k \geq 1\) such that
\begin{equation*}
\sup_{n} \sum_{i,\ell \colon \lambda_{i,\ell}^{(n)} \geq k } \hat f_{n}(i, \ell)^2 < \delta.
\end{equation*}
Then for all \( \delta > 0 \),
\begin{equation*}
\begin{split}
\lim_{\varepsilon \to 0 } \sup_n P\left( f_n(X^{(n)}_{\varepsilon}) \not = f_n(X^{(n)})\right)
&=
2\lim_{\varepsilon \to 0 } \sup_n  \sum_{i, \ell} \left(1-e^{-\varepsilon \lambda_{i, \ell}^{(n)}} \right) \hat f_n(i, \ell)^2
\\&\leq
2\delta + 2\lim_{\varepsilon \to 0 } \sup_n  \sum_{i, \ell \colon \lambda_{i, \ell}^{(n)} < k } \left(1-e^{-\varepsilon k } \right) \hat f_n(i, \ell)^2
\\&\leq
2\delta +2\lim_{\varepsilon \to 0 }  \left(1-e^{-\varepsilon k } \right) 
\\&= 2\delta.
\end{split}
\end{equation*}
As \( \delta \) can be chosen to be arbitrarily small, this implies that \( (f_n)_{n \geq 1} \) is exclusion stable with respect to \( (G_n, \alpha_n)_{n \geq 1 } \).

For the \emph{only if} direction, suppose that there is \( \delta > 0 \) such that for all \( k \geq 1\),
\begin{equation*}
\sup_{n} \sum_{i, \ell \colon \lambda_{i,\ell}^{(n)} \geq k} \hat f_{n}(i,\ell)^2 \geq \delta
\end{equation*}
for all \( k>0\). Then in particular, this is true for \( k = \varepsilon^{-1} \). This implies that
\begin{equation*}
\begin{split}
\lim_{\varepsilon \to 0 } \sup_{n} P \left( f_n(X_0^{(n)}) \not = f_n(X_{\varepsilon}^{(n)}) \right)
&=
2 \lim_{\varepsilon \to 0 } \sup_{n}    \sum_{i,\ell} \left(1-e^{-\varepsilon \lambda_{i, \ell}^{(n)}} \right) \hat f_n(i, \ell)^2
\\&\geq
2 \lim_{\varepsilon \to 0 } \sup_{n}    \sum_{i , \ell \colon \lambda_{i, \ell}^{(n)} \geq \varepsilon^{-1}} \left(1-e^{-\varepsilon \lambda_{i, \ell}^{(n)}} \right) \hat f_n(i, \ell)^2
\\&\geq
2 \lim_{\varepsilon \to 0 } \sup_{n}    \sum_{i , \ell \colon \lambda_{i,\ell}^{(n)} \geq \varepsilon^{-1}} \left(1-e^{-\varepsilon \varepsilon^{-1}} \right) \hat f_n(i, \ell)^2
\\&=
2 \lim_{\varepsilon \to 0 } (1-e^{-1}) \delta.
\end{split}
\end{equation*}
In particular, \( (f_n)_{n \geq 1} \) cannot be noise stable.
\end{proof}

Before ending this section, we present a last lemma which gives an upper bound of the eigenvalues \( \lambda_{i,\ell}^{(n)} \), which more or less follows directly by spelling out the terms in~\eqref{equationIV: rayleigh quotient} and then taking trivial upper bounds.

\begin{lemma}
Let \( Q^{(\ell)} \) be the generator of the symmetric exclusion process on \( G \) with \( \ell \) black marbles and rate \( \alpha \). If \( \lambda \) is an eigenvalue of \( -Q^{(\ell)}\) and \( d \coloneqq \max_{v \in V(G)} \deg v \), then \( \lambda \leq 2\alpha \ell d \).
\label{lemma: bounded eigenvalues}
\end{lemma}

\begin{proof}
First recall the characterization of an eigenvalue \( \lambda_i^{(n,\ell)} \) given by the Rayleigh quotient, namely that for any \( i \),
\[
\lambda_i^{(n, \ell)} \leq \sup_{g \not \equiv 0} \frac{\langle -Q_n^{(\ell)}g,g \rangle}{\langle g,g \rangle}.
\]
Equivalently, for any such eigenvalue,
\[
\lambda_i^{(n, \ell)} \leq \sup_g \frac{\sum_x \pi_n^{(\ell)} (x) g(x) \sum_{y \sim x} \alpha \cdot  (g(x)-g(y)) }{\sum_x \pi_{n}^{(\ell)}(x) g(x)^2} .
\]
Using that \( \pi_n^{(\ell)} \) is uniform, and simplifying, we obtain
\[
\lambda_i^{(n, \ell)} \leq \sup_g \frac{\alpha \sum_{x,y \colon y\sim x}  (g(x)-g(y))^2 }{2 \sum_x g(x)^2} .
\]
As
\[
\sum_{x,y \colon y \sim x} (g(x)-g(y))^2 \leq 2\sum_{x,y \colon y \sim x} \left( g(x)^2 + g(y)^2 \right)
\]
and as for any state \(s \) there is at most \( \ell  d \)  states \( y \) such that \( y \sim x \), we obtain
\[
\lambda_i^{(n, \ell)} \leq \sup_g \frac{4\alpha \sum_{x}   \ell d \, g(x)^2 }{2 \sum_x g(x)^2}  = 2\alpha \ell d.
\]
\end{proof}

\section{Eigenvectors and eigenvalues for symmetric exclusion processes}

Below and in the rest of this section, for any graph \( G \), any \( \ell \in \{ 0,1,\ldots, |V(G)| \} \), any \( x \in \binom{V(G)}{\ell} \) and any \( v \in V(G)\), let \( x_v \) denote the unique element in \( \binom{V(G)}{\ell -1 } \) or \( \binom{V(G)}{\ell+1} \) which differs from \( x \) in only the color at vertex \( v \). Moreover, for any \( e \in E(G) \), let \( x_e \) be the unique element in \( \binom{V(G)}{\ell} \) which is obtained by switching positions of the marbles at the endpoints of \( e \). For any \( v \in V(G) \), let 
\[
x(v) \coloneqq
\begin{cases}
1 & \textnormal{ if the marble at \( v \in V(G)\) is black} \cr
0 & \textnormal{ if the marble at \( v \in V(G)\) is white} 
\end{cases}
\]
and recall that 
\[
\| x \| \coloneqq \left|\{ v \in V(G) \colon x(v)=1 \}\right|.
\]
Finally, for \( m<\ell \) and \( y \in \binom{V(G)}{m} \), let \( y \leq x \) denote that for all \( v \in V(G) \) we have that \( y(v) \leq x(v) \).

Our main objective in this section will be to give the relationships between the eigenvectors \( \{ \psi_i^{(n,\ell)} \}_i\)  and eigenvalues  \( \{ \lambda_i^{(n,\ell)} \}_i\) for different \( \ell \in \{ 1,2, \ldots , \lfloor n/2 \rfloor \} \). This will done by studying the following operators, defined for any eigenvector \( \psi \) of \( Q_n^{(\ell) } \) by
\[
\psi_+ (x) \coloneqq \sum_{v \colon x(v)=0} \psi(x_v), \qquad  x \in \textstyle   \binom{V(G_n)}{\ell-1} 
\]
and
\[
\psi_- (x) \coloneqq \sum_{v \colon x(v)=1} \psi(x_v), \qquad  x \in \textstyle \binom{V(G_n)}{\ell+1} .
\]

The following result will play a major role in the later proof of our main result, Theorem~\ref{theorem:main result}. 

\begin{proposition}
For each \( \ell=1,2,\ldots, \lfloor n/2 \rfloor \), let \( Q_n^{(\ell)} \) be the generator of the symmetric exclusion process with respect to \( (K_n, \alpha, \ell) \). Also, for each such \( \ell \), let \( \{ \psi_i^{(n,\ell)} \}_i \) be an orthonormal set of eigenvectors of  \( -Q_n^{(\ell)} \) and let \( \{ \lambda_i^{(n,\ell)}\}_{i} \) be the corresponding eigenvalues.
Then
\begin{enumerate}
\item[(a)] for each \( j = 1, \ldots, \ell \), the eigenvalue \( \alpha j(n-j+1) \) has multiplicity \
\[ \binom{n}{j} - \binom{n}{j-1}, \]
\item[(b)] If we order the eigenvalues of \( -Q_n^{(\ell)} \) so that 
\[
 0 = \lambda_1^{(n,\ell)} < \lambda_2^{(n,\ell)} \leq \lambda_3^{(n,\ell)} \leq \cdots \leq \lambda_{\binom{n}{\ell}}^{(n,\ell)},
\]
 an orthogonal basis of eigenvectors to \( -Q_n^{(\ell-1)} \) is given by 
\[
 \{  (\psi_i^{(n,\ell)} )_+ \}_{i \in \{ 1,2,\ldots,  \binom{n}{\ell-1} \}}, 
\]
\item[(c)] given the ordering of the eigenvectors given in (b), for \( i \in \{ 1,2,\ldots,  \binom{n}{\ell-1} \} \), we can pick \( \psi_i^{(n,\ell)} = C_{i}( \psi_i^{(n,\ell-1)})_- ,\) for some normalizing constants \( C_{i}>0 \) that depends on \( i \), \( n \), \( \ell \) and \( \alpha \).
\item[(d)] given the ordering of the eigenvectors given in (b), for any \( m < \ell\) we can pick the eigenvectors \( \{ \psi_i^{(n,\ell)} \} \)  of \( -Q_n^{(\ell)} \) with eigenvalue less than or equal to \( \alpha m(n-m+1) \) in such a way that any such eigenvector \( \psi_i^{(n,\ell)} \) can  be written as
\[
\psi_i^{(n,\ell)} (\cdot) = C_i' \sum_{y \leq \cdot \colon \|y\|=m} \psi_i^{(n,m)}(y).
\] 
for normalizing constants \( C_i' > 0 \) that depends on \( i\), \( n \), \( \ell \), \( \alpha \) and \( m \).
Moreover, we have   \( \lambda_i^{(n,\ell)} =  \lambda_i^{(n,m)}.\)
\end{enumerate}
\label{proposition: eigenvectors}
\end{proposition}

\begin{remark}\label{remark: symmetry}
Proposition~\ref{proposition: eigenvectors} is stated only for \( \ell \in \{ 1,2, \ldots , \lfloor n/2 \rfloor \} \), that is for \( \ell \in   \{ 1,2, \ldots ,\lfloor n/2 \rfloor \}  \) black marbles and \( n-\ell \in \{ \lfloor n/2 \rfloor +1, \ldots, n-1 \} \) white marbles. However, as an exclusion process with \( \ell \) black marbles and \( n - \ell \) white marbles essentially behaves in the same way as a exclusion process with \( n-\ell \) black marbles and \( \ell \) white marbles, we have
\[
\psi_i^{(n,\ell)}(x) = \psi_i^{(n,n-\ell)}(1-x)
\]
and
\[
\lambda_i^{(n,\ell)} = \psi_i^{(n,n-\ell)}(1-x).
\]
\end{remark}

The merit of Proposition~\ref{proposition: eigenvectors} is not that it provides new information about the eigenvalues and eigenvectors of the generator of the exclusion process on a complete graph; in fact neither the eigenvectors nor the eigenvalues of this process are unknown (see e.g.~\cite{bgs2013, f22014, bh2012}). Rather, this result is important to us because it provides a quite explicit structure of the eigenvectors, and it is this structure that we will need in the proof of Theorem~\ref{theorem:main result}. In particular, we will need this structural results for exclusion processes on general graphs, which is provided by the next lemma.

For eigenvectors \( \psi \) of \( Q_n^{(1) } \) and \( x \in \binom{V(G_n)}{\ell} \), a function similar to \( \psi_+ \) and \( \psi_- \), defined by
\[
\psi_*(x) \coloneqq \sum_{v \colon x(v)=1} \psi((0,0,\ldots, 0)_v),
\]
was used in~\cite{clr2010} to find bounds on the smallest nonzero eigenvalue of \( -Q^{(\ell)} \) for general graphs. In this paper the authors showed that \( \psi_* \) will be an eigenvector of \( -Q^{(\ell)} \), and thus deduced parts of Proposition~\ref{proposition: eigenvectors} and the lemmas that we will use to prove it. However, they did not extend this definition to eigenvectors eigenvectors \( \psi \) of \( Q_n^{(j) } \) for \( j \geq 1 \). The next lemma therefore extends their result, and provides a way to obtain eigenvectors of \( -Q^{(\ell-1)} \) and \( -Q^{(\ell +1 ) } \) given eigenvectors of \( -Q^{(\ell)}\).

\begin{lemma}
Let \( G \) be a finite connected graph, \( \ell \in \{ 0,1,\ldots, |V(G)|-1 \} \) and let \( \alpha \)  any strictly positive real number. Let \( Q^{(\ell)} \) be the generator of the symmetric exclusion process with respect to  \(  (G, \alpha, \ell)\), and let \( \psi \) be an eigenvector of \( -Q^{(\ell)}\) with eigenvector \( \lambda \). Then, for \( x \in \binom{V(G)}{\ell-1} \), the function \( \psi_+ \colon \binom{V(G)}{\ell-1} \to \{0,1 \}  \) is either an eigenvector to \( -Q^{(\ell-1)} \) with eigenvalue \( \lambda \), or \( \psi_+ \equiv 0 \). Similarly, for \( x \in \binom{V(G)}{\ell+1}\), the function \( \psi_- \colon \binom{V(G)}{\ell+1} \to \{0,1 \}  \) 
is either an eigenvector to \( -Q^{(\ell+1)} \)  with eigenvalue \( \lambda \), or \( \psi_- \equiv 0 \). 
\label{lemma: eigenvectors}
\end{lemma}
Note that by applying the operator \( \psi \mapsto \psi_+ \) \( \ell - m \) times and then using Lemma~\ref{lemma: eigenvectors}, we get the following corollary, which is a weaker version of Proposition~\ref{proposition: eigenvectors} (d) for finite connected graphs.

\begin{corollary}
In the setting of Lemma~\ref{lemma: eigenvectors}, if \( \psi_i^{(m)} \) is an  eigenvector of \( -Q^{(m)} \) with corresponding eigenvalue \( \lambda_i^{(m)} \), then either
\[
\psi(x) \coloneqq \sum_{y \leq x \colon \|y\|=m} \psi_i^{(m)}(y), \qquad x \in \textstyle \binom{V(G)}{\ell}
\]
is a (nonzero) eigenvector of \( -Q_n^{(\ell)} \) with eigenvalue \( \lambda_i^{(\ell)} =  \lambda_i^{(m)}\), or \( \psi \equiv 0 \).

\label{corollary for lemma 4.1}
\end{corollary}

We now  prove Lemma~\ref{lemma: eigenvectors} mainly by spelling out the definitions of \( \psi_+ \) and \( \psi_- \).

\begin{proof}
Note first that for any \( x \in \{ 0,1 \}^{V(G)} \),
\begin{equation*}
\begin{split}
\sum_{e \in E(G)}  \psi_+(x_{e})
&=
\sum_{e \in E(G)}   \sum_{v\in V(G) \colon \atop  x_e(v)=0} \psi((x_{e})_v)
\\&=
\sum_{e \in E(G_)}   \sum_{v\in V(G) \colon \atop  x(v)=0} \psi((x_v)_e).
\end{split}
\end{equation*}

Using this, for \( x \in \binom{V(G)}{\ell-1}  \), we obtain
\begin{equation*}
\begin{split}
-Q^{(\ell-1)}  \psi_+ (x)  
&=  
\alpha \!\!\!\sum_{e \in E(G)} \! (\psi_+(x) - \psi_+(x_{e}))
\\&=
\alpha \!\!\!  \sum_{e \in E(G)} \sum_{v\in V(G) \colon \atop x(v)=0} \! ( \psi(x_v) - \psi((x_v)_{e})
\\&=
\sum_{v\in V(G) \colon \atop x(v)=0} \alpha \!\!\!  \sum_{e \in E(G)} \!  ( \psi(x_v) - \psi((x_v)_{e})
\\&=
 \sum_{v\in V(G) \colon \atop x(v)=0} -Q^{(\ell)}\psi(x_v) 
\\&=
 \sum_{v\in V(G) \colon \atop x(v)=0} \lambda \psi(x_v) 
=
\lambda \psi_+(x).
\end{split}
\end{equation*}
This shows that \( \psi_+ \) is an eigenvector to \( -Q^{(\ell-1)} \) with eigenvalue \( \lambda \), provided that \( \psi_+ \not \equiv 0 \).
Analogously, for \( x \in \binom{V(G)}{\ell+1} \), we have
\begin{equation*}
\sum_{e \in E(G)}  \psi_-(x_{e})
=
\sum_{e \in E(G)}   \sum_{v\in V(G) \colon \atop x_{e}(v)=1} \psi((x_{e})_v)
=
\sum_{e \in E(G)}   \sum_{v\in V(G) \colon \atop  x(v)=1} \psi((x_v)_{e})
\end{equation*}
in turn implying that
\begin{equation*}
\begin{split}
-Q^{(\ell+1)}  \psi_- (x)  
&=  
\alpha  \!\!\!   \sum_{e \in E(G)} (\psi_-(x) - \psi_-(x_{e}))
\\&=
 \alpha \!\!\!  \sum_{e \in E(G)} \sum_{v \colon x(v)=1} ( \psi(x_v) - \psi((x_v)_{e})
\\&=
\sum_{v\in V(G) \colon \atop x(v)=1} -Q^{(\ell)}\psi(x_v) 
\\&=
\sum_{v\in V(G) \colon \atop x(v)=1} \lambda \psi(x_v) 
=
\lambda \psi_-(x).
\end{split}
\end{equation*}
As this shows that \( \psi_- \) is an eigenvector to \( -Q^{(\ell+1)} \) with eigenvalue \( \lambda \) provided that \( \psi_- \not \equiv 0 \), this concludes the proof.
\end{proof}

The purpose of the next lemma is to provide expressions for the lengths of \( \psi_+ \) and \( \psi_-\). In contrast to the previous lemma, this lemma requires that the graph \( G_n \) on which the exclusion process evolves is the complete graph.

\begin{lemma}
Let \( \alpha \) be a positive real number and let \( Q_n^{(\ell)} \) be the generator of the symmetric exclusion process with respect to \( (K_n, \alpha) \). Then for any eigenvector \( \psi \) of \( -Q_n^{(\ell)} \) with corresponding eigenvalue \( \lambda \), 
\[
 \langle \psi_+, \psi_+ \rangle =  \frac{n-\ell+1}{\alpha \ell} \cdot  \left( \alpha\ell(n-\ell+1)  - \lambda\right) 
 \]
 and 
\[
 \langle \psi_-, \psi_- \rangle = \frac{\ell+1}{\alpha(n-\ell)} \cdot  \left(  \alpha(\ell+1)(n-\ell)  - \lambda\right).
 \]
 \label{lemma: length}
\end{lemma}

\begin{proof}
Let \( \psi \) and \( \psi' \) be any two eigenvectors of \( -Q_n^{(\ell)} \). Then by definition,
\begin{equation*}
\begin{split}
\langle  \psi_+ , \psi_+'   \rangle 
&=
\frac{1}{\binom{n}{\ell-1}} \cdot \sum_{x \in \binom{V(K_n)}{\ell-1}}  \psi_+ (x) \psi_+'(x)
\\&=
\frac{1}{\binom{n}{\ell-1}} \cdot \sum_{x \in \binom{V(K_n)}{\ell-1}}     \sum_{v \colon x(v)=0\atop v' \colon x(v')=0} \psi(x_v)   \psi'(x_{v'}).
\end{split}
\end{equation*}
Now for each \( y \in \binom{V(K_n}{\ell} \), in the double sum above, the term \( \psi(y)\psi'(y) \) will be counted \( \ell \) times, as for each \( v  \) such that \( y(v) = 1 \), we can let \( v = v' \), \( x = y_v \) and write \( \psi(y)\psi'(y) \)  as \( \psi(x_{v}) \psi(x_{v'}) \). Similarly, for any \( y,y' \in \binom{V(K_n}{\ell} \) such that \( y \sim y' \), the term \( \psi(y)\psi(y' ) \) will appear exactly one time, as this requires \( x \) to be the configuration with black marbles at the positions where both \( y \) and \( y' \) have black marbles. This implies that
\begin{equation}
\begin{split}
\hspace{3em}&\hspace{-3em}\sum_{x \in \binom{V(K_n)}{\ell-1}}     \sum_{v \colon x(v)=0\atop v' \colon x(v')=0} \psi(x_v)   \psi'(x_{v'})
\\&=
\sum_{x \in \binom{V(K_n)}{\ell}} \Biggl( \ell\, \psi(x)\psi'(x) + \sum_{x' \in \binom{V(K_n)}{\ell} \colon x' \sim  x}\!\!\!\!\!\!\!\!  \psi(x) \psi'(x') \Biggr)
\\&=
 \sum_{x \in \binom{V(K_n)}{\ell}} \psi(x) \Biggl(  \ell\, \psi'(x) +  \sum_{x' \in \binom{V(K_n)}{\ell} \colon x' \sim x} \!\!\!\!\!\!\!\! \psi'(x') \Biggr).
\end{split}
\label{equation: between levels}
\end{equation}
Now recall that for any eigenvector \( \psi' \) of \( -Q_n^{(\ell)} \) with corresponding eigenvalue \( \lambda \), and any \( x \in \binom{V(K_n)}{\ell} \),
\[
\lambda \psi'(x) = -Q_n^{(\ell)} \psi' = \alpha \ell (n-l) \psi'(x) \,-\,  \alpha \!\!\!\! \sum_{x'\in \binom{V(K_n)}{\ell} \colon x'\sim x}  \!\!\!\! \psi'(x').
\]
Using this, we obtain
\begin{equation}
\begin{split}
\langle  \psi_+ , \psi_+'   \rangle 
&=
\frac{1}{\binom{n}{\ell-1}} \cdot  \sum_{x \in \binom{V(K_n)}{\ell}} \psi(x) \Biggl(   \ell\, \psi'(x) +  \ell (n-\ell) \psi'(x)  -\alpha^{-1} \lambda \psi'(x) \Biggr)
\\&=
\frac{1}{\binom{n}{\ell-1}}  \cdot  \frac{ \alpha \ell(n-\ell+1)  - \lambda}{\alpha} \cdot  \sum_{x \in \binom{V(K_n)}{\ell}} \psi(x) \psi'(x)
\\&=
\frac{n-\ell+1}{\ell}  \cdot  \frac{ \alpha \ell(n-\ell+1)  - \lambda}{\alpha} \cdot   \langle \psi, \psi' \rangle.
\end{split}
\label{equation: plus equation}
\end{equation}
If we set \( \psi = \psi' \), the first claim of the lemma immediately follows.

To repeat the argument with \( \psi_- \) and \( \psi_-' \) instead of \(\psi_+ \) and \( \psi_+' \), the only thing we need to replace is~\eqref{equation: between levels}. By a similar argument as in the first case, we obtain
\begin{equation*}
\begin{split}
\hspace{3em}&\hspace{-3em}\sum_{x \in \binom{V(K_n)}{\ell+1}}     \sum_{v \colon x(v)=1\atop v' \colon x(v')=1} \psi(x_v)   \psi'(x_{v'})
\\&=
\sum_{x \in \binom{V(K_n)}{\ell}} \Biggl( (n-\ell)\, \psi(x)\psi'(x) + \sum_{x' \in \binom{V(K_n)}{\ell} \colon x' \sim  x}\!\!\!\!\!\!\!\!  \psi(x) \psi'(x') \Biggr)
\end{split}
\end{equation*}
Using this equation, we get the following equation.
\begin{equation}
\langle  \psi_- , \psi_-'   \rangle 
=
\frac{\ell+1}{n-\ell}  \cdot  \frac{ \alpha(\ell+1)(n-\ell)  - \lambda}{\alpha} \cdot   \langle \psi, \psi' \rangle.
\label{equation: minus equation}
\end{equation}
If we set \( \psi = \psi' \), the second claim of the lemma now immediately follows.

\end{proof}

The next lemma shows that orthogonality is preserved by the operations \( \psi \mapsto \psi_+ \) and \( \psi \mapsto \psi_-\).

\begin{lemma}
Let \( \alpha \) be a positive real number and let \( Q_n^{(\ell)} \) be the generator of the symmetric exclusion process with respect to \( (K_n, \alpha) \). Then for any two orthogonal eigenvectors \( \psi \) and \( \psi' \) of \( -Q_n^{(\ell)} \), 
\[
 \langle \psi_+, \psi'_+ \rangle =  
 \langle \psi_-, \psi'_- \rangle = 0.
 \]
\label{lemma: orthogonality}
\end{lemma}

\begin{proof}
As \( \psi \) and \( \psi' \) are orthogonal, we have that \( \langle \psi, \psi' \rangle = 0 \). Using~\eqref{equation: plus equation}, we obtain
\begin{equation*}
\langle  \psi_+ , \psi_+'   \rangle 
=
\frac{n-\ell+1}{\ell}  \cdot  \frac{ \alpha \ell(n-\ell+1)  - \lambda}{\alpha} \cdot   \langle \psi, \psi' \rangle = 0.
\end{equation*}
Analogously, using~\eqref{equation: minus equation}, we obtain
\begin{equation*}
\langle  \psi_- , \psi_-'   \rangle 
=
\frac{\ell+1}{n-\ell}  \cdot  \frac{ \alpha(\ell+1)(n-\ell)  - \lambda}{\alpha} \cdot   \langle \psi, \psi' \rangle = 0.
\end{equation*}
\end{proof}

We are now ready to give a proof of Proposition~\ref{proposition: eigenvectors}.

\begin{proof}[Proof of Proposition~\ref{proposition: eigenvectors}]
We will first prove that (a), (b) and (c) hold by using induction on the number of black marbles, \( \ell \). As an induction hypothesis, suppose that for some \( \ell \in \mathbb{N} \), there is an orthonormal basis of eigenvectors \( \psi_1^{(n, \ell)}\), \ldots, \(\psi_{\binom{n}{\ell}}^{(n, \ell)} \) to \( -Q_n^{(\ell)}\) with corresponding eigenvalues
\[
\lambda_i^{(n, \ell)} = \begin{cases}
0 &\textnormal{for } i = 1 \cr
\alpha j(n-j+1) &\textnormal{for } \binom{n}{j} < i \leq \binom{n}{j+1},\, j=0,1,\ldots ,\ell.
\end{cases}
\]

By Lemmas~\ref{lemma: eigenvectors} and~\ref{lemma: orthogonality}, the nonzero vectors among \( \psi_{1-}^{(n, \ell)}\), \ldots, \(\psi_{\binom{n}{\ell}-}^{(n, \ell)} \) is an orthogonal set of eigenvectors of \( -Q_n^{(\ell+1)}\) with corresponding eigenvalues \( \lambda_1^{(n, \ell)} \), \ldots, \(  \lambda_{\binom{n}{\ell}}^{(n, \ell)} \). By Lemma~\ref{lemma: length} and the induction hypothesis, for any \( i \in \{ 1,2, \ldots, \binom{n}{\ell} \} \),
\[
 \langle \psi_{i-}^{(n, \ell)}, \psi_{i-}^{(n,\ell)} \rangle = \frac{\ell+1}{\alpha(n-\ell)} \cdot  \left(  \alpha(\ell+1)(n-\ell)  - \lambda_{i}^{(n, \ell)}\right) \not = 0.
 \]
For \( i \in \{ 1, \ldots, \binom{n}{\ell}\} \), set 
\[
 \psi_i^{(n, \ell+1)} \coloneqq \frac{\psi_{i-}^{(n, \ell)}}{\sqrt{ \langle \psi_{i-}^{(n, \ell)}, \psi_{i-}^{(n,\ell)} \rangle}}.
\]
Then for each \( i \in \{ 1,2,\ldots, \binom{n}{\ell} \}\), \( \psi_i^{(n, \ell+1)}  \) is an eigenvector of \( -Q_n^{(\ell+1) } \) with corresponding eigenvalue \( \lambda_i^{(n, \ell+1)} = \lambda_i^{(n, \ell)}\). Moreover, we can extend the set  \( \{ \psi_i^{(n, \ell+1)} \}_{i=1, \ldots, \binom{n}{\ell}} \) to an orthonormal basis \( \{ \psi_i^{(n, \ell+1)} \}_{i=1, \ldots, \binom{n}{\ell+1}} \) of eigenvectors of \( -Q_{n}^{(\ell+1)} \). To show that the induction hypothesis must hold for \( \ell +1 \) black marbles given that it holds for \( \ell \) black marbles, it now suffices to show that \( \lambda_i^{(n, \ell+1)} = \alpha(\ell+1)(n - \ell) \) for all \( i > \binom{n}{\ell} \). To this end, note that by Lemmas~\ref{lemma: eigenvectors} and~\ref{lemma: orthogonality}, the nonzero vectors in the set \( \{ \psi_{i+}^{(n, \ell+1)} \}_{i=1, \ldots, \binom{n}{\ell+1}} \) are an orthogonal set of eigenvectors of \( -Q_{\ell}^{(n)} \). By Lemma~\ref{lemma: length}, 
\[
 \langle \psi_{i+}^{(n, \ell+1)}, \psi_{i+}^{(n, \ell+1)} \rangle =  \frac{n-\ell}{\alpha (\ell+1)} \cdot  \left( \alpha(\ell+1)(n-\ell)  - \lambda_{i}^{(n, \ell+1)}\right).
 \]
By the induction hypothesis, this is nonzero for \( i \in \{ 1,  2, \ldots, \binom{n}{\ell} \}\). As no orthogonal set of eigenvectors of \( -Q_n^{(\ell)} \) can contain more than \( \binom{n}{\ell} \) elements, we must have that \( \lambda_i^{(n, \ell+1)} = \alpha(\ell+1)(n-\ell) \) for all \( i > \binom{n}{\ell} \). As the induction hypothesis is well known to hold for \( \ell = 0 \), the desired conclusion follows.

(d) follows directly from (a), (b) and (c) by applying the operator \( \psi \mapsto \psi_+ \) \( \ell - m \) times.
\end{proof}

\section{A proof  of  Theorem~\ref{theorem:main result}}
Before we give a proof of our first main result, Theorem~\ref{theorem:main result}, we will prove the following lemma, which is interesting in itself, relating the eigenvectors of an exclusion process on any graph with the eigenvectors of an exclusion process on the complete graph.

\begin{lemma}
Let \( Q_n^{(\ell)} \) be the generator of the symmetric exclusion process on \( K_n \) with \( \ell \) black marbles and rate \( \alpha \), and let \(R_n^{(\ell)} \) be the generator of the symmetric exclusion process of a graph \( G_n \), \( V(G_n) = V(K_n) \), with \( \ell \) black marbles with rate \( \beta \). 
Let \( \{ \psi_i^{(n, \ell)}\}_i \) be the eigenvectors of \( -Q_n^{(\ell)} \), and let \( \{ \lambda_i^{(n,\ell)} \}_i \)  be the corresponding eigenvalues.  Analogously, let   \( \{ \chi_i^{(n, \ell)} \}_i \) be the eigenvectors of \( -R_n^{(\ell)} \), and let \( \{ \mu_i^{(n,\ell)} \}_i \)  be the corresponding eigenvalues. 
Further, let \( d = \max_{v \in V(G_n)} \deg v \). Then for any \( k  \) and   \( k' \) such that \( \alpha k'(n-k'+1) \geq k \),
\[
\Span_{i \colon \lambda_i^{(n,l)} \leq k} \psi_i^{(n,\ell)} \subseteq \Span_{i' \colon \mu_{i'}^{(n,l)} \leq 2\beta k' d} \chi_{i'}^{(n,\ell)} .
\]

Consequently, if \( Q_n \) is the generator of the symmetric exclusion process on \( K_n \) with rate \( \alpha \),  \(R_n \) is the generator of the symmetric exclusion process of \( G_n \) with rate \( \beta\) and \( \{ \psi^{(n)}_{i, \ell} \} \) and \( \{ \chi^{(n)}_{i, \ell} \} \) are the orthonormal bases of eigenvectors of \( -Q_n \) and \( -R_n \) respectively, as defined in Lemma~\ref{lemma: general eigenvectors}, then 
\[
\Span_{i \colon \lambda_{i, \ell}^{(n)} \leq k} \psi_{i, \ell}^{(n)} \subseteq \Span_{i' \colon \mu_{i', \ell}^{(n)} \leq 2\beta k' d} \chi_{i',\ell}^{(n)}
\]
whenever \( \alpha k'(n-k'+1) \geq k \).
\label{lemma: spectral containment}
\end{lemma}

When we use Lemma~\ref{lemma: spectral containment} in the proof of Theorem~\ref{theorem:main result}, we will think of \( n \) as being very large and \( k \) as being small and fixed, and pick \( \alpha = 1/n \) and \( \beta = 1/d \). With this choice of parameters, and any \( k\) and \( k' \) such that \(  k'(n-k'+1)/n \geq k \), Lemma~\ref{lemma: spectral containment} says that
\[
\Span_{i \colon \lambda_{i, \ell}^{(n)} \leq k} \psi_{i, \ell}^{(n)} \subseteq \Span_{i' \colon \mu_{i', \ell}^{(n)} \leq  2k'} \chi_{i',\ell}^{(n)}.
\]
From the simple inequality
\[
 \frac{x(n-x+1)}{n} \geq \frac{n+1}{2n} \cdot x,
\]
valid for \( x \in [0,n/2] \), we obtain that in this special case, we can choose any \( k' \geq \frac{2n}{n+1}  \cdot k \). In particular, we can choose \( k' = 2k \). From this we get the following lemma as a corollary.

\begin{lemma}
%
In the setting of Lemma~\ref{lemma: spectral containment}, if \( \alpha = 1/n \), \( \beta = 1/d \) and \( k \leq n/4 \), then
\[
\Span_{i \colon \lambda_{i, \ell}^{(n)} \leq k} \psi_{i, \ell}^{(n)} \subseteq \Span_{i \colon \mu_{i, \ell}^{(n)} \leq 4k} \chi_{i,\ell}^{(n)}.
\]
\label{lemma: spectral containment for the main result}
\end{lemma}

\begin{proof}[Proof of Lemma~\ref{lemma: spectral containment}]
Note first that by Remark~\ref{remark: symmetry}, it suffices to prove the lemma in the case \( \ell \leq \lfloor n/2 \rfloor \)
Fix \( k  \) and let  \( \psi_i^{(n, \ell)} \) be an eigenvalue of \( -Q_n^{(\ell)}\) with corresponding eigenvalue \( \lambda_i^{(n, \ell)} \leq k\). 
By Proposition~\ref{proposition: eigenvectors} (a), \( \lambda_i^{(n, \ell)} = \alpha j(n-j+1) \) for some \( j\in \{ 1,2, \ldots, \ell \} \). 
By Proposition~\ref{proposition: eigenvectors} (d), this in turn implies that   \( \psi_i^{(n, \ell)} \) can be written as
\[
\psi_i^{(n, \ell)} = \psi_i^{(n, \ell)}(\cdot) = C \sum_{y \leq \cdot \colon \| y \| = j} \psi_i^{(n,j)}(y)
\]
for some normalizing constant \( C \).
As \( \psi_i^{(n,j)} \colon \binom{V(G_n)}{j} \to \mathbb{R}\),  \( V(G_n) = V(K_n) \)  and \( \{ \chi_i^{(n, j)} \}_{i} \) is an orthonormal basis for all functions \( f \colon \binom{V(K_n)}{j} \to \mathbb{R}\),
\[
\psi_i^{(n,j)} \in \Span_{i'}  \chi_{i'}^{(n,j)} .
\]
By Lemma~\ref{lemma: bounded eigenvalues}, this is equivalent to that
\[
\psi_i^{(n,j)} \in   \Span_{i' \colon \mu_{i'}^{(n,j)}\leq 2\beta j d}  \chi_{i'}^{(n,j)}
\]
implying that
\begin{equation*}
\begin{split}
\psi_i^{(n,\ell)} = C \sum_{y \leq \cdot \colon \| y \| = j} \psi_i^{(n,j)}(y) &\in \Span_{{i'} \colon \mu_{i'}^{(n,j)}\leq 2\beta j  d}  \sum_{y \leq \cdot \colon \| y \| = j}  \chi_{i'}^{(n,j)} (y) 
\\&\subseteq \Span_{i' \colon \mu_{i'}^{(n,\ell)}\leq 2\beta j d}   \chi_{i'}^{(n,\ell)}.
\end{split}
\end{equation*}
where the last inclusion follows from  Corollary~\ref{corollary for lemma 4.1}.
Now  as \( \alpha k'(n-k'+1)  \geq k \geq \alpha j(n-j+1) \) and \(j \leq \ell \leq \lfloor n/2 \rfloor \), we have that \( j \leq k' \). Using this, we obtain
\begin{equation*}
\begin{split}
\psi_i^{(n,\ell)} \in   \Span_{i' \colon \mu_{i'}^{(n,\ell)}\leq 2\beta j d}   \chi_{i'}^{(n,\ell)} \subseteq \Span_{i' \colon \mu_{i'}^{(n,\ell)}\leq 2\beta k' d}   \chi_{i'}^{(n,\ell)}
\end{split}
\end{equation*}
which is the desired conclusion.

\end{proof}

We are now ready to give a proof of Theorem~\ref{theorem:main result}. The main idea of this proof is to use Lemma~\ref{lemma: spectral containment for the main result} to compare the sums in the characterizations of exclusion sensitivity and exclusions stability given by Propositions~\ref{proposition: spectral XS}~and~\ref{proposition: spectral XStability} for the two sequences of graphs.

\begin{proof}[Proof of Theorem~\ref{theorem:main result}]
Note first that it is enough to prove the result for \( \beta_n = 1/\max_{v \in V(G_n)} \deg v \).

Suppose that \( (f_n)_{n \geq 1 } \) is not exclusion sensitive with respect to the sequence \( (K_{|V(G_n)|}, 1/|V(G_n)| )_{n \geq 1}\). By Proposition~\ref{proposition: spectral XS}, either 
\[
\lim_{n \to \infty}\Var (\E[f_n(x) \mid \| x \| = \| X_0^{(n)} \| ]) \not = 0 
\]
or there is \( k > 0 \), \( \varepsilon>0 \) and a subsequence \( n' \) such that 
\begin{equation}\label{equation: length of projection}
\sum_{i, \ell \colon 0 < \lambda_{i, \ell}^{(n')} \leq k } \hat f_{n'}(i, \ell)^2  > \varepsilon
\end{equation}
for all \( n' \).
In the first case, we are already done, so we can assume that~\eqref{equation: length of projection} holds.
Now~\eqref{equation: length of projection} says exactly that the length of the projection of \( f_{n'} \) onto \( \Span_{i, \ell \colon 0< \lambda_{i, \ell}^{(n')} \leq k} \psi_{i,\ell}^{(n')}\) is at least \( \varepsilon \). By Lemma~\ref{lemma: spectral containment for the main result}, 
\[
\Span_{i \colon \lambda_{i, \ell}^{(n)} \leq k} \psi_{i, \ell}^{(n)} \subseteq \Span_{i \colon \mu_{i, \ell}^{(n)} \leq 4k} \chi_{i,\ell}^{(n)}.
\]
As \( \pi^{(n')} \) is the uniform measure for both \( G_n \) and \( K_{|V(G_n)|} \), these two spaces have the same inner product. This implies that the length of the projection onto the larger of the two spaces must be larger than the length of the projection onto the smaller subspace. In other words, if we define \( \check f_{n'}(i,\ell) \coloneqq \langle f_{n'}, \chi_{i,\ell}^{(n')} \rangle \) then   we must have
\begin{equation*}
\sum_{i, \ell \colon 0 < \mu_{i, \ell}^{(n')} \leq 4k } \check f_{n'}(i, \ell)^2
\;\;\;  \geq \!\!\!\!  \sum_{i, \ell \colon 0 < \lambda_{i, \ell}^{(n')} \leq k } \hat f_{n'}(i, \ell)^2  \;\; > \;\; \varepsilon
\end{equation*}
for all \( n' \). From this it follows that \( (f_n)_{n \geq 1} \) cannot be exclusion sensitive with respect to \( (G_n, 1/d_n)_{n \geq 1 } \), and finishes the proof of (i). 

To show that (ii) holds, suppose that \( (f_n)_{n \geq 1} \) is exclusion stable with respect to \( (K_{|V(G_n)|}, 1/|V(G_n)| )_{ n\geq 1} \). Then, by Proposition~\ref{proposition: spectral XStability}, for all \( \delta > 0 \) there is \( k >0 \) such that
\begin{equation*}
\sup_{n} \sum_{i, \ell \colon \lambda_{i,\ell}^{(n)} \geq k} \hat f_{n}(i,\ell)^2 < \delta
\end{equation*}
or equivalently, such that
\begin{equation*}
\inf_{n} \sum_{i, \ell \colon \lambda_{i,\ell}^{(n)} < k} \hat f_{n}(i,\ell)^2 > \langle f, f \rangle - \delta.
\end{equation*}
By Lemma~\ref{lemma: spectral containment for the main result}, this implies that
\begin{equation*}
\inf_{n} \sum_{i, \ell \colon \mu_{i,\ell}^{(n)} < 4k} \check f_{n}(i,\ell)^2 > \langle f, f \rangle- \delta.
\end{equation*}
As \( \delta \) was arbitrary,  by Proposition~\ref{proposition: spectral XStability}, \( (f_n )_{n \geq 1} \) is exclusion stable with respect to \( (G_n, 1/d_n)_{n\geq 1} \). This finishes the proof.
\end{proof}

\section{Monotonicity at equal rate}

The main purpose of this section is to give a proof of Theorem~\ref{proposition: monotonicity}, which gave conditions given which the properties of being exclusion sensitive and exclusion stable was monotone with respect to adding edges to a sequence of graphs. We now formulate   the main lemma we will use in the proof of this result.

\begin{lemma} 
Let \( Q\) be the generator for the symmetric exclusion process with respect to \( (G, \alpha) \) and \( Q' \) be the generator for the symmetric exclusion process with respect to \( (G', \alpha)\), for two finite connected graphs \( G \) and \( G' \) with the same number of vertices and a strictly positive real number \( \alpha \). Let \( \{ \psi_{i,\ell} \}_{i,\ell} \) be and orthonormal set of eigenvectors of \( -Q \) with corresponding eigenvalues \( \{ \lambda_{i,\ell} \}_{i,\ell}\), and let \( \{ \chi_{i,\ell} \}_{i,\ell} \) and \( \{ \mu_{i,\ell}\}_{i,\ell} \) be the corresponding sets for \( -Q' \). Further, let \( f \colon \{ 0,1 \}^{V(G)} \to \mathbb{R} \). Then, if \( E(G') \subseteq E(G) \), for all strictly positive real numbers \( k \) and \( k' \)  we have that
\begin{equation*}
\sum_{i,\ell \colon \mu_{i,\ell} > k'} \langle f, \chi_{i,\ell}  \rangle^2
\leq 
\left(\sqrt{  \frac{k}{k'}  \sum_{i,\ell \colon 0<\lambda_{i,\ell}\leq k} \langle f, \psi_{i,\ell}\rangle^2 } + \sqrt{\sum_{i,\ell \colon \lambda_{i,\ell} > k} \langle f, \psi_{i,\ell} \rangle^2}\right)^2\!\!\!.
\end{equation*}

\label{lemma: monotonicity inequality}
\end{lemma}

The main idea of our proof of this lemma is to  use~\eqref{equation: Q inequality} to relate the eigenvectors and eigenvalues of \( -Q \) and \( -Q' \). This gives good bounds for functions with support only on eigenvectors corresponding to small eigenvalues, why the function \( f \) is first split into two parts; one of which is the projection of \( f \) onto the span of such  eigenvectors. The squares and square roots arises naturally by an application of the triangle inequality.

\begin{proof}

Fix \( k> 0 \) and \( k' > 0 \) and define \( P_{\lambda \leq k} f \coloneqq  \sum_{i, \ell \colon 0<\lambda_{i, \ell}  \leq k} \langle f, \psi_{i,\ell} \rangle \psi_{i,\ell} \) to be the projection of \( f \) onto the space spanned by all eigenvectors \( \psi_{i, \ell}\) with corresponding eigenvalue less than or equal to \(  k \) but not equal to zero. Similarly, define \( P_{\lambda > k} f \coloneqq  \sum_{i, \ell \colon \lambda_{i, \ell}  >k} \langle f, \psi_{i,\ell} \rangle \psi_{i,\ell} \).
Then for any \( k'>0 \),
\begin{equation*}
\begin{split}
 {\sum_{i , \ell \colon \mu_{i,\ell} > k'} \langle f, \chi_{i,\ell} \rangle^2}
&= {\sum_{i , \ell \colon  \mu_{i,\ell} > k'} \Bigl\langle  P_{\lambda \leq k} f + P_{\lambda > k} f + \sum_{j, m \colon  \lambda_{j,m} = 0} \langle f, \psi_{j,m} \rangle \psi_{j,m} , \, \chi_{i,\ell}  \Bigr\rangle^2}
\\&= {\sum_{i , \ell \colon \mu_{i,\ell} > k'} \Bigl\langle  P_{\lambda \leq k} f + P_{\lambda > k} f +  \sum_{j,m \colon \mu_{j,m} = 0} \langle f, \chi_{j,m} \rangle \chi_{j,m}, \, \chi_{i,\ell}  \Bigr\rangle^2}
\end{split}
\end{equation*}
where the last equality follows from Remark~\ref{remark: equal eigenvectors}. 
Using that any eigenvector with corresponding eigenvalue equal to zero is orthogonal to any eigenvector \( \chi_{i, \ell} \) with corresponding eigenvalue \( \mu_{i, \ell }\geq k'>0 \), and then applying the triangle inequality, we obtain
\begin{align}
 {\sum_{i , \ell \colon \mu_{i,\ell} > k'} \langle f, \chi_{i,\ell} \rangle^2}
&= 
{\sum_{i , \ell \colon \mu_{i,\ell} > k'} \langle P_{\lambda \leq k} f + P_{\lambda > k} f, \chi_{i,\ell}  \rangle^2}\nonumber
\\[-1ex]&\leq
 \left( \sqrt{ {\sum_{i , \ell\colon \mu_{i,\ell} > k'} \langle P_{\lambda \leq k} f , \chi_{i,\ell} \rangle^2 }}
+
\sqrt{{\sum_{i , \ell\colon \mu_{i,\ell} > k'}  \langle  P_{\lambda > k} f, \chi_{i,\ell} \rangle^2}} \right)^{2}\!\!\!.\label{equation line: two terms}
\end{align}
Now note that
\begin{align*} 
\sum_{i, \ell \colon \mu_{i,\ell}> k'} \langle  P_{\lambda \leq k} f  , \chi_{i,\ell}\rangle^2 
&\leq
\sum_{i, \ell \colon \mu_{i,\ell}> k'}\frac{\mu_{i,\ell}}{k'}  \cdot \langle    P_{\lambda \leq k} f  , \chi_{i,\ell}\rangle^2 
&\leq
\frac{1}{k'}  \sum_{i, \ell \colon \mu_{i,\ell}> 0} \mu_{i,\ell} \cdot \langle    P_{\lambda \leq k} f  , \chi_{i,\ell}\rangle^2 
\\
&=
\frac{1}{k'} \cdot \langle -Q' P_{\lambda \leq k} f, P_{\lambda \leq k} f \rangle.
\end{align*}
Applying~\eqref{equation: Q inequality}, it follows that
\begin{align*} 
\sum_{i, \ell \colon \mu_{i,\ell}> k'} \langle  P_{\lambda \leq k} f  , \chi_{i,\ell}\rangle^2 
&\leq
\frac{1}{k'} \cdot \langle -Q P_{\lambda \leq k} f, P_{\lambda \leq k} f \rangle
=
\frac{1}{k'}  \sum_{i,\ell \colon \lambda_{i,\ell} > 0} \lambda_{i,\ell} \cdot \langle P_{\lambda \leq k} f, \psi_{i,\ell}\rangle^2
\\&=
\frac{1}{k'}  \sum_{i,\ell \colon 0<\lambda_{i,\ell} \leq k} \lambda_{i,\ell} \cdot \langle   f, \psi_{i,\ell}\rangle^2
\leq
\frac{k}{k'}   \sum_{i,\ell \colon 0<\lambda_{i,\ell} \leq k}   \langle   f, \psi_{i,\ell}\rangle^2.
\end{align*}

For the second term in the last expression of~\eqref{equation line: two terms}, again using Remark~\ref{remark: equal eigenvectors}, we have
\begin{align*}
\sum_{i \colon \mu_{i,\ell} > k'} \langle   P_{\lambda > k} f, \chi_{i,\ell} \rangle^2
&\leq 
\sum_{i \colon \mu_{i,\ell} > 0} \langle   P_{\lambda > k} f, \chi_{i,\ell} \rangle^2
= 
\sum_{i \colon \lambda_{i,\ell} > 0} \langle   P_{\lambda > k} f, \psi_{i,\ell} \rangle^2
\\
&= 
\sum_{i \colon \lambda_{i,\ell} > k} \langle   f, \psi_{i,\ell} \rangle^2
\end{align*}
Summing up, we now get,
\begin{equation*}
\sum_{i,\ell \colon \mu_{i,\ell} > k'} \langle f, \chi_{i,\ell} \rangle^2
\leq 
\left(\sqrt{  \frac{k}{k'}  \sum_{i,\ell \colon 0< \lambda_{i,\ell}\leq k} \langle f, \psi_{i,\ell}\rangle^2 } + \sqrt{\sum_{i,\ell \colon \lambda_{i,\ell} > k} \langle f, \psi_{i,\ell} \rangle^2}\right)^2\!\!\!.
\end{equation*}
which is the desired conclusion.
\end{proof}

We now give a proof of Theorem~\ref{proposition: monotonicity}, whose conclusion will follow more or less directly by applying Lemma~\ref{lemma: monotonicity inequality} to the  sums in the characterizations of exclusions sensitivity and exclusion stability given by Propositions~\ref{proposition: spectral XS}~and~\ref{proposition: spectral XStability}.

\begin{proof}[Proof of Theorem~\ref{proposition: monotonicity}]

Let \( \hat f_n(i, \ell) = \langle f_n, \psi_{i,\ell}^{(n)} \rangle \) and \( \check f_n(i,\ell) = \langle f_n, \chi_{i,\ell}^{(n)} \rangle \).

For the proof of the first part of the theorem, suppose that \( (f_n)_{n \geq 1} \) is exclusion stable with respect to \( (G_n, \alpha_n)_{n \geq 1} \). Then by Proposition~\ref{proposition: spectral XStability}, for all \( \delta > 0 \) there is \( k>0 \) such that \( \sum_{i,\ell \colon \lambda_{i,\ell}^{(n)} > k} \hat f_n(i,\ell) < \delta/4 \) for all \( n\geq 1 \).
Since 
\[
 \sum_{i,\ell \colon 0< \lambda_{i,\ell}^{(n)} \leq k} \hat f_n(i,\ell)^2 \leq  \sum_{i,\ell } \hat f_n(i,\ell)^2 = \E [f_n(X_0^{(n)})^2] \leq 1
\]
 there is \( k'>0 \) such that
\[
\frac{k}{k'}  \sum_{i,\ell \colon 0< \lambda_{i,\ell}^{(n)} \leq k} \hat f_n(i,\ell)^2 < \delta/4
\] 
for all \( n \). Using Lemma~\ref{lemma: monotonicity inequality}, we thus obtain
\begin{align*}
\sum_{i, \ell \colon \mu_{i,\ell}^{(n)} > k'}  \check f_n(i,\ell)^2
&\leq 
\left(\sqrt{\frac{k}{k'}\cdot { \sum_{i, \ell \colon 0<\lambda_{i,\ell}^{(n)} \leq k}  \hat f_n(i,\ell)^2    } } + \sqrt{ \sum_{i, \ell \colon \lambda_{i,\ell}^{(n)} > k}  \hat f_n(i,\ell)^2 }\right)^2
\\& \leq \left(\sqrt{\delta/4} + \sqrt{\delta/4}\right)^2
=
\delta.
\end{align*}
As \( \delta \) was arbitrary, by Proposition~\ref{proposition: spectral XStability}, \( (f_n)_{n \geq 1} \) is exclusion stable with respect to \( (G_n', \alpha_n)_{n \geq 1} \).

For the other direction, suppose that \( (f_n)_{n \geq 1} \) is exclusion sensitive with respect to \( (G_n', \alpha_n)_{n \geq 1} \). By Proposition~\ref{proposition: spectral XS},
\[
\lim_{n \to \infty}\Var (\E[f_n(x) \mid \| x \| = \| X_0^{(n)} \| ]) = 0 
\]
and for all \( k' > 0 \),
\begin{equation}
\lim_{n \to \infty} \sum_{i, \ell \colon 0< \mu_{i,\ell}^{(n)} \leq  k'}  \check f_n(i,\ell)^2 = 0.
\label{equation: limit for one of the graph sequences}
\end{equation}

By Lemma~\ref{lemma: monotonicity inequality},  for any \( k> 0 \) and \(k' > 0 \)  we have that
\begin{align*}
 \sum_{i, \ell \colon 0< \mu_{i,\ell}^{(n)} \leq  k'}  \check f_n(i,\ell)^2 
 &= 
 \langle f_n, f_n \rangle -  \sum_{i, \ell \colon  \mu_{i,\ell}^{(n)} = 0}  \check f_n(i,\ell)^2 -  \sum_{i, \ell \colon  \mu_{i,\ell}^{(n)} > k'}  \check f_n(i,\ell)^2 
 \\
 &= 
 \langle f_n, f_n \rangle -  \sum_{i, \ell \colon  \lambda_{i,\ell}^{(n)} = 0}  \hat f_n(i,\ell)^2 -  \sum_{i, \ell \colon  \mu_{i,\ell}^{(n)} > k'}  \check f_n(i,\ell)^2 
 \\&\geq
  \langle f_n, f_n \rangle -  \sum_{i, \ell \colon  \lambda_{i,\ell}^{(n)} = 0}  \hat f_n(i,\ell)^2 
- \left(\sqrt{\frac{k}{k'}\cdot {  \sum_{i,\ell \colon 0< \lambda_{i,\ell}^{(n)} \leq k} \hat f_n(i,\ell)^2 } } + \sqrt{  \sum_{i,\ell \colon \lambda_{i,\ell}^{(n)} > k} \hat f_n(i,\ell)^2 }\right)^2.
\end{align*}
Using~\eqref{equation: limit for one of the graph sequences}, it thus follows that  
\begin{align*}
0 \geq \limsup_{n \to \infty}  \left( \langle f_n, f_n \rangle -  \sum_{i, \ell \colon  \lambda_{i,\ell}^{(n)} = 0}  \hat f_n(i,\ell)^2 
- \left(\sqrt{\frac{k}{k'}\cdot {  \sum_{i,\ell \colon 0< \lambda_{i,\ell}^{(n)} \leq k} \hat f_n(i,\ell)^2 } } + \sqrt{  \sum_{i,\ell \colon \lambda_{i,\ell}^{(n)} > k} \hat f_n(i,\ell)^2 }\right)^2 \right).
\end{align*}
As this holds for any \( k'> 0 \) and
\[
 \sum_{i,\ell \colon 0< \lambda_{i,\ell}^{(n)} \leq k} \hat f_n(i,\ell)^2 \leq  \sum_{i,\ell} \hat f_n(i,\ell)^2 = \E [f_n^2] \leq 1
\]
we obtain
\begin{align*}
0 
&\geq 
\limsup_{n \to \infty}  \left(  \langle f_n, f_n \rangle -  \sum_{i, \ell \colon  \lambda_{i,\ell}^{(n)} = 0}  \hat f_n(i,\ell)^2 
-   \sum_{i,\ell \colon \lambda_{i,\ell}^{(n)} > k} \hat f_n(i,\ell)^2 \right)
\\&=
 \limsup_{n \to \infty}   \sum_{i, \ell \colon  0<\lambda_{i,\ell}^{(n)} \leq k' }  \hat f_n(i,\ell)^2,
\end{align*}
which in particular implies that for any \( k > 0 \),
\begin{align*}
 \limsup_{n \to \infty}   \sum_{i, \ell \colon  0<\lambda_{i,\ell}^{(n)} \leq k' }  \hat f_n(i,\ell)^2 =0 .
\end{align*} 
Proposition~\ref{proposition: spectral XS} now ensures that \( (f_n)_{n\geq 1} \) is exclusion sensitive with respect to \( (G_n, \alpha_n)_{n \geq 1} \).
\end{proof}

\begin{remark}\label{remark: different rates}
The proof of Theorem~\ref{proposition: monotonicity} is easy to extend to the setting where the rates \( (\alpha_n)_{n \geq 1} \) is allowed to be different for different edges in the graphs, as long as the same edge has the same rate in both graphs.
To see this, simply note that the actual rates \( \alpha_n \) was never used in the proof, which depends only on Proposition~\ref{proposition: spectral XS}, Proposition~\ref{proposition: spectral XStability} and Lemma~\ref{lemma: monotonicity inequality}, which in turn only uses the earlier Remark~\ref{remark: equal eigenvectors} and~\eqref{equation: Q inequality}. All of these results can easily be seen to be valid also in this setting.
\end{remark}

\begin{remark}\label{remark: disconnected} 
Using the previous remark, we can quite easily make Theorem~\ref{proposition: monotonicity} even more general. Suppose namely that we are in the setting of Theorem~\ref{proposition: monotonicity}, except that the graphs \( (G_n')_{n \geq 1 } \) are not necessarily connected, but that for each \( n \geq 1 \), \( G_n' \) is the union of \( c_n < |V(G_n)|\) connected components. For \( n\geq 1 \) define intermediate graphs \( G_n^{(1)}\),  \( G_n^{(2)}\) and \( G_n^{(3)} \) as follows
\begin{itemize}
\item Let \( G_n^{(1)} \) be a graph with \( V(G_n^{(1)}) = V(G_n) \) and \( E(G_n') \subset E(G_n^{(1)}) \subseteq E(G_n) \) and where removing any edge in \( E(G_n^{(1)}) \backslash E(G_n') \) would make \( G_n^{(1)} \) disconnected. Call such a set of edges a minimal connecting set of edges for \( G_n \), and note that the number of edges in such a set will always be \( c_n-1 \).
\item Let \( G_n^{(2)} \) be a graph with \( V(G_n^{(2)}) = V(G_n) \) and \(  E(G_n^{(1)}) \subseteq E(G_n^{(2 )})  \), where \( E(G_n^{(2 )})  \backslash E(G_n^{(1 )}) \) is another minimal connecting set of edges for \( G_n \), and let the edges  in this set all have rate \(((c_n-1) n)^{-1}\). Note that we do not necessarily have that \( E(G_n^{(2)}) \subseteq E(G_n) \).
\item  Let \( G_n^{(3)} \) be the graph with \( V(G_n^{(3)}) = V(G_n) \) and \( E(G_n^{(3)}) = E(G_n')  \cup (E(G_n^{(2)}) \backslash E(G_n^{(1)})  ) \). Note that \( G_n^{(3)} \) is connected.
\end{itemize}

Now let \( f_n \colon \{ 0,1\}^{V(G_n)} \to \{ 0,1 \} \) ,    \( E_n \subseteq E(G_n) \) and write \( \mathcal{E}_{n,t} \) for the event that no edge in \( E_n \) was used before time \( t \). Then 
\begin{align*}
P(f_n(X_0^{(n)} ) \not = f_n(X_\varepsilon^{(n)})\mid \mathcal{E}_{n,\varepsilon})  \cdot P(\mathcal{E}_{n, \varepsilon})
&\leq
P(f_n(X_0^{(n)} ) \not = f_n(X_\varepsilon^{(n)})) 
\\&\leq
P(f_n(X_0^{(n)} ) \not = f_n(X_\varepsilon^{(n)}) \mid \mathcal{E}_{n,\varepsilon}) 
+   P(\mathcal{E}_{n,\varepsilon}^c).
\end{align*}
For the covariance, we get
\begin{align*}
&\Cov(f_n(X_0^{(n)} ), f_n(X_1^{(n)})) 
\\& \qquad  =
\E[f_n(X_0^{(n)} )  f_n(X_1^{(n)}) ]  - \E[f_n(X_0^{(n)} ) ]^2
\\& \qquad \leq
\E[f_n(X_0^{(n)} )  f_n(X_1^{(n)})  \mid \mathcal{E}_{n,1} ]  + P(\mathcal{E}_{n,1}^c )- \E[f_n(X_0^{(n)} ) \mid \mathcal{E}_{n,1} ]^2 \cdot P(\mathcal{E}_{n,1} )^2
\\& \qquad  =
\Cov (f_n(X_0^{(n)} ) , f_n(X_1^{(n)})  \mid \mathcal{E}_{n,1} ) + P(\mathcal{E}_{n,1}^c ) +  \E[f_n(X_0^{(n)} ) \mid \mathcal{E}_{n,1} ]^2 \cdot (1-P(\mathcal{E}_{n,1} )^2)
\\& \qquad \leq
\Cov (f_n(X_0^{(n)} ) , f_n(X_1^{(n)})  \mid \mathcal{E}_{n,1} ) + 3P(\mathcal{E}_{n,1}^c ) 
\end{align*}
and similarly,
\begin{align*}
&\Cov(f_n(X_0^{(n)} ), f_n(X_1^{(n)})) 
\\& \qquad  =
\E[f_n(X_0^{(n)} )  f_n(X_1^{(n)}) ]  - \E[f_n(X_0^{(n)} ) ]^2
\\& \qquad \geq
\E[f_n(X_0^{(n)} )  f_n(X_1^{(n)}) \mid \mathcal{E}_{n,1}] \cdot P(\mathcal{E}_{n,1}) - \left( \E[f_n(X_0^{(n)} ) \mid \mathcal{E}_{n,1} ] \cdot {P(\mathcal{E}_{n,1})} + P(\mathcal{E}_{n,1}^c)\right) ^2
\\& \qquad =
\Cov( f_n(X_0^{(n)} ) , f_n(X_1^{(n)}) \mid \mathcal{E}_{n,1}) \cdot P(\mathcal{E}_{n,1}) + \E[f_n(X_0^{(n)} ) \mid \mathcal{E}_{n,1} ]^2 \cdot P(\mathcal{E}_{n,1} ) P(\mathcal{E}_{n,1}^c)
\\& \qquad \qquad 
- P(\mathcal{E}_{n,1}^c)^2 -2 P(\mathcal{E}_{n,1}^c) \cdot {P(\mathcal{E}_{n,1})} \cdot \E[f_n(X_0^{(n)} ) \mid \mathcal{E}_{n,1} ] 
\\& \qquad \geq
\Cov( f_n(X_0^{(n)} ) , f_n(X_1^{(n)}) \mid \mathcal{E}_{n,1}) \cdot P(\mathcal{E}_{n,1})
-3 P(\mathcal{E}_{n,1}^c) 
\\& \qquad \geq
\Cov( f_n(X_0^{(n)} ) , f_n(X_1^{(n)}) \mid \mathcal{E}_{n,1}) 
-4 P(\mathcal{E}_{n,1}^c) .
\end{align*}

As for \( E_n =  (E(G_n^{(2)}) \backslash E(G_n^{(1)})  ) \) we have
\[
P(\mathcal{E}_{n,1}) \leq P(\mathcal{E}_{n,\varepsilon}) = \left( e^{-\varepsilon ((c_n-1)n)^{-1}} \right)^{c_n-1} =  e^{-\varepsilon /n},
\]
and
\[
\limsup_{n \to \infty} e^{-\varepsilon /n} = 1
\]
these inequalities allow us to transfer the properties of being exclusion sensitive and exclusion stable between sequences of graphs that differ only on sets of edges that are being used very rarely in the limit.
Using this, we get the series of implications
\begin{align*}
\text{XS on } G_n' \Rightarrow \text{XS on } G_n^{(3)} \Rightarrow \text{XS on } G_n^{(2)} \Rightarrow \text{XS on } G_n^{(1)} \Rightarrow \text{XS on } G_n 
\end{align*}
where the second and fourth implication uses  Theorem~\ref{proposition: monotonicity} and Remark~\ref{remark: different rates}.

Analogously for exclusion stability, we have
\begin{align*}
\text{XStable on } G_n \Rightarrow \text{XStable on } G_n^{(1)} \Rightarrow \text{XStable on } G_n^{(2)} \Rightarrow \text{XStable on } G_n^{(3)} \Rightarrow \text{XStable on } G_n'
\end{align*}
where again, the second and fourth implication uses  Theorem~\ref{proposition: monotonicity} and Remark~\ref{remark: different rates}.

This shows that the assumption on that \( G_n' \) is to be connected can be dropped from Theorem~\ref{proposition: monotonicity}.
A similar argument shows that also the assumption that \( G_n \) is connected for every \( n \) can be dropped.
\end{remark}

\bibliographystyle{plain}
\bibliography{../../Litteratur/references}

\begin{thebibliography}{10}

\bibitem{abgm2014}
Daniel Ahlberg, Erik Broman, Simon Griggiths, and Robert Morris.
\newblock Noise sensitivity in continuum percolation.
\newblock {\em Israel Journal of Mathematics}, 201:847--899, 2014.

\bibitem{schramm2000}
Itai Benjamini, Gil Kalai, and Oded Schramm.
\newblock Noise sensitivity of {B}oolean functions and applications to
  percolation.
\newblock {\em Publications Mathématiques de l'Institut des Hautes Études
  Scientifiques}, 90:5--43, 1999.

\bibitem{bgs2013}
Erik Broman, Christophe Garban, and Jeffrey Steif.
\newblock Exclusion sensitivity of {B}oolean functions.
\newblock {\em Probability Theory and Related Fields}, 155:621--663, 2013.

\bibitem{bh2012}
Andries~E. Brouwer and Willem~H. Haemers.
\newblock {\em Spectra of graphs}.
\newblock Springer, 2012.

\bibitem{clr2010}
Pietro Caputo, Thomas~M. Liggett, and Thomas Richthammer.
\newblock Proof of aldous' spectral gap conjecture.
\newblock {\em Journal of the American Mathematical Society}, 23:831--851,
  2010.

\bibitem{chung1996}
Fan Chung.
\newblock {\em Spectral graph Theory}.
\newblock Number~92 in CBMS Regional Conference Series in Mathematics. American
  Mathematical Society, 1997.

\bibitem{f22014}
Yuval Filmus.
\newblock Orthogonal basis for functions over a slice of the {B}oolean
  hypercube.
\newblock {\em The Electronic Journal of Combinatorics}, 23(1):--, - 2016.
\newblock arXiv:1406.0142.

\bibitem{gs2014}
Christophe Garban and Jeffrey Steif.
\newblock {\em Noise Sensitivity of Boolean Functions and Percolation}.
\newblock Cambridge University Press, first edition, 2014.

\bibitem{kd2012}
Guy Kindler and Ryan O'Donnell.
\newblock Gaussian noise sensitivity and {F}ourier tails.
\newblock In {\em 2012 IEEE 27th Conference on Computational Complexity}, 2012.

\bibitem{md2005}
Elchanan Mossel and Ryan O'Donnel.
\newblock Coin flipping from a cosmic source: On error correction of truly
  random bits.
\newblock {\em Random Structures and Algorithms}, 26(4):418--436, 2005.

\bibitem{st1999}
Oded Schramm and Boris Tsirelson.
\newblock Trees, not cubes: hypercontractivity, cosiness and noise stability.
\newblock {\em Electronic communications in probability}, 4:39--49, 1999.

\end{thebibliography}

\end{document}